\documentclass[11pt,french,oneside]{amsart}
 \usepackage{amssymb}
\usepackage[margin=1in]{geometry}
			\usepackage{amsmath,amscd}
            \usepackage{fontenc}
            \usepackage{textcomp}
                   
\usepackage{inputenc}
		 \theoremstyle{plain}
      \newtheorem{theorem}{Theorem}[section]
      \newtheorem{lemma}[theorem]{Lemma}
       
      \newtheorem*{conjecture}{Conjecture}
      \newtheorem{coro}[theorem]{Corollary}
      
			\newtheorem*{theoremA}{Theorem A}
			\newtheorem*{theoremD}{Theorem D}

			\newtheorem*{coroB}{Corollary B}
			\newtheorem*{coroC}{Corollary C}
      \newtheorem{prop}[theorem]{Proposition}
      \newtheorem{defi}[theorem]{Definition}

       \numberwithin{equation}{section}

      \theoremstyle{plain}
      \newtheorem{remark}[theorem]{Remark}
\begin{document}

%



  \author{Adrien Boyer and Dustin Mayeda} 
   \thanks{ adrien.boyer@weizmann.ac.il, Weizmann institute of Science, Rehovot, Israel.\\ dmayeda@tx.technion.ac.il, Technion, Haifa, Israel.\\
   This work is supported by ERC Grant 306706. }

   \title{Equidistribution, Ergodicity and Irreducibility associated with Gibbs measures}

   \begin{abstract}
 We generalize an equidistribution theorem \`a la Bader-Muchnik (\cite{BM}) for operator-valued measures constructed from a family of boundary representations associated with Gibbs measures in the context of convex cocompact discrete group of isometries of a simply connected connected Riemannian manifold with pinched negative curvature. We combine a functional analytic tool, namely the property RD of hyperbolic groups (\cite{Jo} and \cite{Ha}), together with a dynamical tool: an equidistribution theorem of Paulin, Pollicott and Schapira inspired by a result of Roblin (\cite{Ro}). In particular, we deduce irreducibility of these new classes of boundary representations. 
    \end{abstract}




   \dedicatory{}



   \maketitle

\section{Introduction}

Viewing the group $SL(2,\mathbb{R})$ as a group acting by isometries of the hyperbolic plane we have an induced action on the geometric boundary of the hyperbolic plane which is identified with the circle.  This boundary action is quasi-invariant with respect to Lebesgue measure (i.e. the sets of Lebesgue measure zero are preserved under the action) and so there is a naturally associated unitary representation of $SL(2,\mathbb{R})$ on $L^2(\mathbb{S}^1)$ called the quasi-regular representation.  The quasi-regular representation is irreducible and is part of a family of irreducible unitary representations of $SL(2,\mathbb{R})$ on $L^2(\mathbb{S}^1)$ called the principal series which forms one of the three families comprising the unitary dual.  For a general locally compact group $G$, especially when $G$ is a discrete countable group, there is no hope of computing its unitary dual so we will restrict ourselves to the problem of determining when the associated quasi-regular representation of a $G$ quasi-invariant action is irreducible.  From a dynamical viewpoint the associated quasi-regular representation is interesting because it reflects the ergodic theoretic properties of the action such as ergodicity and mixing.

Early on, Furstenberg \cite{F} showed that when $G$ is a semisimple Lie group the space $G/P$ equipped with Haar measure where $P$ is a minimal parabolic subgroup, nowadays called the Poisson-Furstenberg boundary,   can be realized as the Poisson boundary of a random walk on a lattice in $G$. 
Motivated by these results we further restrict the problem which we state as the following conjecture of Bader and Muchnik \cite{BM} :

\begin{conjecture}
For a locally compact group $G$ and a spread-out probability measure $\mu$ on $G$, the quasi-regular representation associated to a $\mu$-boundary of $G$ is irreducible.
\end{conjecture}

For the rest of the paper we will restrict ourselves to the case when $G$ is a discrete countable group.  Analogously to the case of $SL(2,\mathbb{R})$ the action of the free group $\mathbb{F}_n$ on its boundary is quasi-invariant with respect to the Patterson-Sullivan measure class and thus there is the associated quasi-regular representation.  Fig\`a-Talamanca and Picardello (see \cite{F1} and \cite{F3}) construct the analog of the principal series which are unitary representations of $\mathbb{F}_n$ on $L^2(\partial \mathbb{F}_n)$ and show they are all irreducible.  For homogeneous trees Fig\`a-Talamanca and Steger \cite{F2} show similar irreducibility results for lattices in the automorphism group.  Kuhn and Steger \cite{KS} have also constructed different examples of irreducible representations of the free group.  The conjecture has also been solved for some actions of simple algebraic groups by Bekka and Cowling in \cite{BC}.  When $G$ is a lattice in a Lie group Cowling and Steger \cite{CS} showed that the irreducible representations of the ambient semisimple Lie group restricted to $G$ remain irreducible.  In particular the quasi-regular representation
of $SL(2,\mathbb{R})$ on $L^2(\mathbb{S}^1)$ restricted to lattices is irreducible.  Later on in the context of CAT(-1) spaces for which a discrete group of isometries $G$ acts cocompactly, Connell and Muchnik (see \cite{CM} and \cite{CM2}) proved when the geometric boundary is equipped with a certain class of Gibbs measures that it can be realized as the Poisson boundary of a random walk in $G$.  This result led Bader and Muchnik \cite{BM} to prove the conjecture for the action of the fundamental group of a compact negatively curved manifold on the geometric boundary of the universal cover of the manifold, endowed with the Patterson-Sullivan measure class. Recently the first named author has also generalized the main theorem of Bader and Muchnik in \cite{BM} to the context of CAT(-1) spaces and so irreducibility of boundary representations associated with Patterson-Sullivan measures. Moreover Garncarek \cite{LG} has generalized the irreducibility result of \cite{BM} for the action of a Gromov-Hyperbolic group on its geometric boundary endowed with the Patterson-Sullivan measure class. He has also deduced thanks to the work of \cite{BMH} that if a symmetric random walk on a Gromov-Hyperbolic group has finite exponential moment with respect to a word metric and such that the associated Green metric satisfies the \emph{Ancona inequality} then the action on the Poisson boundary with respect to the harmonic measure is irreducible thanks to the work of \cite{BHM}.  It is not clear at all to us if any of the measures constructed in \cite{CM2} have finite exponential moment and satisfies the\emph{ Ancona inequality} and therefore it is not clear at all that our result of irreducibility would follow from an application of these two results. Hence it legitimates our dynamical approach to prove irreducibility of quasi-regular representations associated with the class of measure arising as conditional measures of the \emph{Gibbs measures}, called \emph{Gibbs streams} in \cite{CM2} or also called \emph{Patterson densities} in \cite{Pau}; generalizing the Patterson-Sullivan measures class.

Bader and Muchnik prove in \cite[Theorem 3]{BM} an equidistribution theorem for some operator-valued measures associated with Patterson-Sullivan measures. This theorem can be thought of as a generalization of Birkhoff-von Neumann's ergodic theorem for quasi-invariant measures for fundamental groups acting on the geometric boundary of universal covers of compact negatively curved manifolds endowed with the Patterson-Sullivan measures. These quasi-regular representations are called \textit{boundary representations}. It turns out that the irreducibility of boundary representations follows from this generalization of Birkhoff-von Neumann's ergodic theorem.  
  
We generalize the results of Bader and Muchnik to the action of a convex cocompact discrete subgroup of isometries of a pinched negatively curved manifold on its boundary endowed with the \emph{Gibbs streams} measure class  rather called in this paper \emph{Patterson densities} measure class.  The Patterson densities are constructed by first assigning a weight to each element of the orbit and then proceeding as in the construction of the Patterson-Sullivan measures which we think of as the unweighted case.

Historically it was Sinai who first merged the field of equilibrium statistical mechanics from which the concept of Patterson density is imported from with the field of hyperbolic smooth dynamical systems.  Given a H\"older-continuous potential $F$ on the unit tangent bundle of a compact negatively curved manifold, the pressure of $F$ associated with the geodesic flow is given by

$$P(F) = \sup\limits_{m}  \Bigg\{ h_{m} - \int F dm \Bigg\}, $$

where the supremum is taken over all measures on the unit tangent bundle which are invariant under the fundamental group of the manifold and the geodesic flow and $h_{m}$ is the metric entropy of $m$ associated with the geodesic flow.  Bowen \cite{B} proved for negatively curved manifolds that there exists a unique measure called the Gibbs measure which achieves the supremum and is in fact the eigenmeasure associated to the transfer operator of $F$.  As we said, the \emph{Patterson densities} arise as conditional measures of the Gibbs measure and when $F = 0$ the Patterson densities are the Patterson-Sullivan measures and the Gibbs measure is the Bowen-Margulis-Sullivan measure that maximizes the entropy.

  The main tools of this paper are the property RD (Rapid Decay) that hyperbolic groups satisfy (see \cite{Ha} and \cite{Jo}) combined with a spectral characterization of the amenability of the action on the boundary (see \cite{K} and \cite{ADR}) together with an equidistribution theorem of Paulin-Pollicott-Schapira inspired by Roblin's equidistribution theorem which is itself based on the mixing property of the geodesic flow. Indeed this idea of using the mixing property of the geodesic flow goes back to Margulis \cite{M} who used it in order to count the closed geodesics on compact negatively curved manifolds. However, the first object to understand is the \emph{Harish-Chandra function} associated with Patterson densities. This function plays a  major role in harmonic analysis of spherical functions and in the theory of irreducible representations of semismple Lie groups, see for example \cite{GV}.

 \subsection*{Notation }\label{results}

Let $M$ be a complete connected Riemannian manifold with pinched negative curvature. Let $X=\widetilde{M}$, let $q : X \rightarrow M$ be a universal Riemannian covering map
with a covering group $\Gamma$ viewed as a non-elementary discrete group of isometries of $X$, denote the sphere at infinity by $\partial X$ and endow  $\overline{X}=X\cup \partial X$ with the cone topology. 

The limit set of $\Gamma$ denoted by $\Lambda_{\Gamma}$ is the set of all accumulation points in $\partial X$ of an orbit. Namely $\Lambda_{\Gamma}:=\overline{\Gamma x}\cap \partial X$, with the closure in $\overline{X}$. Notice that the limit set does not depend on the choice of $x\in X$. We denote by $\Omega\Gamma$ the subset of $T^{1}X$ of tangent vectors to the geodesic lines in $X$ whose endpoints both lie in $\Lambda_\Gamma$. Following the notations in \cite{CM}, define the geodesic hull $GH(\Lambda_{\Gamma})$ as the union of all geodesics in $X$ with both endpoints in $\Lambda_{\Gamma}$. The convex hull of $\Lambda_{\Gamma}$ denoted by $CH(\Lambda_{\Gamma})$, is the smallest convex subset of $X$ containing $GH(\Lambda_{\Gamma})$. In CAT(-1) spaces we always have $CH(\Lambda_{\Gamma})=GH(\Lambda_{\Gamma})$. We say that $\Gamma$ is \emph{convex cocompact} if it acts cocompactly on  $CH(\Lambda_{\Gamma})$.

Let $p:T^{1}X\rightarrow X$ be the base point projection map from the unit tangent bundle to $X$. Let $g = (g_{t})_{t\in \mathbb{R}}$ be the geodesic flow on $T^{1}M$ and $ \widetilde{g}= (\widetilde{g}_{t})_{t\in \mathbb{R}}$ the one on $T^{1}X$ and equip the unit tangent bundle with the following metric $$d_{T^{1}X}({\rm v},{\rm w})=\frac{1}{\sqrt{\pi} }\int_{\mathbb{R}}d\big(p(g_{t} ({\rm v})),p(g_{t}({\rm w}))\big){\rm e}^{-t^{2}/2}dt,$$ where we use the notation ${\rm v}$ for an element in $T^{1}X$ (and $v$ for an element of $\partial X$).\\   Let $F :
T^{1}M \rightarrow \mathbb{R}$ be a H\"older-continuous map, called a \emph{potential}, and let $\widetilde{F} = F \circ q$ be the $\Gamma$-invariant potential associated on $T^{1}X$. In this work, as it has been suggested by Kaimanovich, we assume
 that \emph{$\widetilde{F}$ is symmetric}, that is $\widetilde{F}$ is invariant by the antipodal map 
 \begin{equation}
 \iota: {\rm v }\in T^{1}X \mapsto - {\rm v}\in T^{1}X.  
 \end{equation} 
 For all  $x,y \in X$, let us define $$ \int_{x}^{y}\widetilde{F}: =
\int_{0}^{d(x,y)} \widetilde{F}(\widetilde{g}_{t}\big({\rm v})\big) {\rm d}t $$
where ${\rm v}=(x,\vec{v}_{x,y})\in T^{1}X$ and $\vec{v}_{x,y}$ is the
unit tangent vector at $x$ to a geodesic from $x$ through $y$.
Set:
	\begin{equation}\label{distanceF}
	d^{F}(x,y):=\int_{x}^{y}\widetilde{F}.
	\end{equation}
A priori $d^{F}$ is not non-negative and is far to be a distance, neverthelss
the symmetry of $\widetilde{F}$ implies
\begin{equation} \label{sym}
d^{F}(x,y)=d^{F}(y,x).
\end{equation}

Define the Gibbs cocycle as
\begin{equation}\label{gibbs}
C^{F}_{v}(x,y):=\lim_{t\rightarrow +\infty} \int_{y}^{v_{t}}\widetilde{F}-\int_{x}^{v_{t}}\widetilde{F}=\lim_{t\rightarrow +\infty} d^{F}(y,v_{t})-d^{F}(x,v_{t}),
\end{equation}
where $v_{t}$ is any geodesic ray ending at a point $v$ in $\partial X$. 
Observe that if $\widetilde{F}=-1$ the Gibbs cocycle is nothing else than the Busemann cocycle, that is the horospherical distance from $x$ to $y$ relative to $v$.

The foundations of Patterson-Sullivan measures theory are in the important papers \cite{Pa}, \cite{Su}. See \cite{Bou},\cite{BMo}, and \cite{Ro} for more general results in the context of CAT(-1) spaces. These measures are also called \textit{conformal densities}. In this paper we are dealing with the \emph{Patterson density of $(\Gamma,F)$} where $F$ is the potential function defined above and $\Gamma$ a discrete group of isometries of $X$.

Recall that $\gamma_{*}\nu$ means $\gamma_{*}\nu(B)=\nu(\gamma^{-1}B)$ where $\gamma$ is in $\Gamma$ and $B$ is a Borel subset of some measure space. 
More specifically we say that \emph{$\nu^{F}$ is a Patterson density of dimension $\sigma \in \mathbb{R}$ for $(\Gamma,F)$} if $\nu^{F}$ is a map which satisfies the following conditions:
\begin{itemize}
	\item $\nu^{F}$ is a map from $x\in X \mapsto \nu_{x}^F \in M(\overline{X})$, i.e. $\nu^{F}_{x}$ is a positive finite measure.
	\item For all $x$ and $y$ in $X$, $\nu_{x}^F$ and $\nu^{F}_{y}$ are equivalent, and we have $$\frac{d\nu^{F}_{y}}{d\nu^{F}_{x}}(v)={\rm e}^{{C^{F-\sigma}_{v}(x,y)}}.$$ 
	\item For all $\gamma \in \Gamma$, and for all $x \in X$ we have $\gamma_{*}\nu^{F}_{x}=\nu^{F}_{\gamma x}$.
\end{itemize}

	In this context define the critical exponent of $(\Gamma,F)$ for $c>0$ large enough as $$\sigma_{\Gamma,F}:=\limsup \frac{1}{n}\sum_{n-c\leq d(\gamma x,x)\leq n}{\rm e}^{{d^{F}(x,\gamma x)}} .$$ 
	
	Even if the construction of a Patterson density was not done in this general context with a potential function, the technique is exactly the same. We attribute the following proposition to Patterson in his seminal paper \cite{Pa}, ensuring the existence of a Patterson density:
	\begin{prop}(S-J. Patterson)
	If $\sigma_{\Gamma,F}<\infty$, then there exists at least one Patterson density of dimension $\sigma_{\Gamma,F}$ with support exactly equal to $\Lambda_{\Gamma}$.
	\end{prop}

 A Patterson density $\nu^{F}$ of dimension $\sigma$ gives rise to a unitary representations $(\pi_{\nu^{F}_x})_{x\in X}$ defined for $x\in X$ as:
 
	\[
		\pi_{\nu^{F}_x}:\Gamma\to \mathcal{U}\big(L^2(\partial X,\nu^{F}_x) \big)
	\]
	\begin{equation}\label{boundaryrep}
	   (\pi_{\nu^{F}_x}(\gamma)\xi)(v)=  {\rm e}^{{\frac{1}{2}} {C^{F-\sigma}_{v}(x,\gamma x)}}  \xi(\gamma^{-1}v),
	   	\end{equation}
where $\xi \in L^{2}(\partial X,\nu^{F}_{x})$ and $v\in \partial X$.

The representations $(\pi_{\nu^{F}_{x}})_{x\in X}$ are unitarily equivalent.
Let $x$ be in $X$ and denote $\pi_{\nu^{F}_x}$ by $\pi_{x}$.
 The matrix coefficient 
\begin{equation}\label{HCHfunction}
\phi_{x}: \Gamma \rightarrow \langle \pi_{x}(\gamma)\mathbf{1}_{\partial X}, \mathbf{1}_{\partial X}\rangle \in \mathbb{R}^{+},
\end{equation}  is called the \emph{Harish-Chandra function}, where $\mathbf{1}_{\partial X}$ denotes the characteristic function of $\partial X$.
\\

\subsection*{Construction of ergodic operator-valued measures}
The Banach space of bounded linear operators from the Banach space of continuous functions to the Banach space of bounded operators on a Hilbert space will be denoted by $\mathcal{L}\big(C(Z),\mathcal{B}(\mathcal{H})\big)$. 
The Banach space $\mathcal{L}\big(C(Z),\mathcal{B}(\mathcal{H})\big)$ is naturally isomorphic to the dual of the Banach space $C(Z)\widehat{\otimes} \mathcal{H} \widehat{\otimes} \overline{\mathcal{H}}$ where $\widehat{\otimes}$ denotes the projective tensor product: Thus $\mathcal{L}\big(C(Z),\mathcal{B}(\mathcal{H})\big)$ will be called the space of \textit{operator-valued measures}.\\
 Pick $x$ in $X$, and a positive real number $\rho $ and define for all integers $ n\geq1 $ the annulus $$C_{n}(x)=\lbrace \gamma \in \Gamma | n-1\leq d(\gamma x, x) < n \rbrace .$$ Let $D_{y}$ be the unit Dirac mass centered at a point $y\in X$. Consider the sequence of operator-valued measures defined for all integers $n \geq 1$  as:
\begin{equation}\label{suitesmesures} 
 \mathcal{M}^{n}_{x} :f\in C(\overline{X}) \mapsto  c_{\Gamma,F} {\rm e}^{-\sigma_{\Gamma,F} n}\sum_{\gamma \in C_{n}(x)} {\rm e}^{d^{F}(x,\gamma x)}D_{\gamma x} (f) \frac{\pi_{x}(\gamma)}{\phi_{x}(\gamma)} \in \mathcal{B}\big(L^{2}(\partial X,\nu^{F}_{x}) \big),
  \end{equation}
  with the normalization constant
  \begin{equation}\label{constante}
   c_{\Gamma,F}=\frac{\sigma_{\Gamma,F} \|m_{F}\|}{1-{\rm e}^{-\sigma_{\Gamma,F}}},
   \end{equation}
where $\|m_{F}\|$ is the mass of the so-called \emph{Gibbs} measure associated with $\nu^{F}$. We refer to Section \ref{Gibbs} for definitions and properties of Gibbs measures. The normalization constant $c_{\Gamma,F}$ ensures  that for any $x $ in $X$ $$c_{\Gamma,F} {\rm e}^{-\sigma_{\Gamma,F} n}\sum_{\gamma \in C_{n}(x)} {\rm e}^{d^{F}(x,\gamma x)}D_{\gamma x} \otimes D_{\gamma^{-1}y}\rightharpoonup \nu^{F}_{x}\otimes \nu^{F}_{y},$$ as n goes to $+\infty$ with respect to the weak$^*$ convergence on $C(\overline{X})^{*}$. \\

 If $f\in C(\overline{X})$, we denote by $f_{|_{\partial X}}$ its continuous restriction to the space $\partial X$.
 
 Let $m(f)$ be the operator in $ \mathcal{B} \big(L^{2}(\partial X,\nu^{F}_{x}) \big)$ acting on $L^{2}(\partial X,\nu^{F}_{x})$ by multiplication and define the operator-valued measure $\mathcal{M}_{x}$  as: 
 \begin{equation}\label{mesurelimite}
\mathcal{M}_{x}: f\in C(\overline{X}) 
 \mapsto m(f_{|_{\partial X}})P_{\textbf{1}_{\partial X}}  \in  \mathcal{B} \big(L^{2}(\partial X,\nu^{F}_{x}) \big).
\end{equation}
where $P_{\textbf{1}_{\partial X}}$ denotes the orthogonal projection on the space of constant functions.
\subsection*{Main Results}\label{results}

 The main result of this paper is the following theorem:

\begin{theoremA}\label{maintheorem} (Equidistribution \`a la Bader-Muchnik) \\Let $\Gamma$ be a convex cocompact discrete group of isometries of a complete connected Riemannian manifold with pinched negative curvature $X$. Let $\widetilde{F}: T^{1}X \rightarrow \mathbb{R}$ be a H\"older-continuous $\Gamma$-invariant potential and let $\nu^{F}$ be a Patterson density for $(\Gamma,F)$ of dimension $\sigma_{\Gamma,F}$.\\
Assume that $\widetilde{F}$ is \emph{symmetric} and  assume that the Gibbs measure associated with $\nu^F$ is \emph{mixing} with respect to the geodesic flow. Then for each $x$ in $ CH(\Lambda_{\Gamma})$ we have  $$\mathcal{M}^{n}_{x} \rightharpoonup \mathcal{M}_{x}$$ as $n \rightarrow + \infty$ with respect to the weak* topology of the Banach space $\mathcal{L}\big(C(\overline{X}),\mathcal{B}(L^{2}(\partial X,\nu^{F}_{x}) )\big)$.
In other words we have for all $f\in C(\overline{X})$ and all $\xi$, $\eta \in L^{2}(\partial X,\nu^{F}_{x})$:
$$\lim_{n\rightarrow +\infty}\langle \mathcal{M}^{n}_{x}(f)\xi,\eta \rangle= \bigg(\int_{\partial X} \xi d\nu_{x}^{F}\bigg) \bigg(\int_{\partial X} f_{|_{\partial X}}\overline{\eta}d\nu_{x}^{F}\bigg).$$
\end{theoremA}

With the same hypotheses of the above theorem, we deduce immediately an ergodic theorem \`a la Birkhoff-von Neumann for the Patterson density $(\Gamma,F)$ associated with  $\nu^{F}_{x}$ on $\partial X$.

\begin{coroB}\label{maincoro}(Ergodicity \`a la Birkhoff-von Neumann)\\
 For all $x\in CH(\Lambda_{\Gamma})$ $$c_{\Gamma,F} { \rm e}^{-\sigma_{\Gamma,F} n}\sum_{\gamma \in C_{n}(x)}{\rm e}^{d^{F}(x,\gamma x)} \frac{\pi_{x}(\gamma)}{\phi_{x}(\gamma)} \rightarrow  P_{\textbf{1}_{\partial X}}$$ as $n\rightarrow +\infty$ with respect to the weak operator topology on $\mathcal{B}(L^{2}(\partial X,\nu^{F}_{x}))$. 
\end{coroB}

In the same setting of Theorem A we have:
\begin{coroC}\label{maincoro2}(Irreducibility)
Assume that $\widetilde{F}$ is \emph{cohomologous to a symmetric potential} and assume that the Gibbs measure is \emph{mixing}. For all $x\in X$, the representations $\pi_{x}:\Gamma \rightarrow \mathcal{U}(L^{2}(\partial X, \nu^{F}_{x}))$ are irreducible.
\end{coroC}
\begin{remark}
The assumption of mixing of Gibbs measures with respect to the geodesic flow is automatic in the case of constant curvature, hence all boundary representations of convex cocompact groups associated with a Patterson density with a  H\"older-continuous potential $\widetilde{F}$  cohomologous to a symmetric potential is irreducible.  Note that this property does not depend on the base point.
\end{remark}

We obtain also the following theorem which classifies the unitary representations associated with a Patterson density. We refer to Subsection \ref{cohomo'} for the definitions concerning Item $(3)$ and $(4)$.
\begin{theoremD}Let $\Gamma$ be a convex cocompact discrete group of isometries of $X$, pick a point $x$ in $X$ and let $\nu_{x}^{F}$ and  $\nu_{x}^{G}$ be Patterson densities associated with H\"older-continuous  $\Gamma$-invariant \emph{symmetric} potentials $\widetilde{F}$ and $\widetilde{G}$ on $T^{1}X$.  Assume that the Gibbs measure is \emph{mixing} with respect to the geodesic flow.
Then the following assertions are equivalent: 
\begin{enumerate}
\item The unitary $\pi_{\nu_{x}^{F}}$ and $\pi_{\nu_{x}^{G}}$ are equivalent as unitary representations.
\item The measures $\nu_{x}^{F}$ and $\nu_{x}^{G}$ are in the same class.
\item The potentials $\widetilde{F}$ and $\widetilde{G}$ have the same periods.
\item The Gibbs cocycles associated with $F$ and $G$ are cohomologous in restriction to $\Omega \Gamma$.

\end{enumerate}
\end{theoremD}
The method of proof of Theorem A consists of two steps: given a sequence of functionals of the dual of a separable Banach space, we shall prove:
 \begin{itemize}
 \item \textbf{Step 1:} The sequence is uniformly bounded: existence of accumulation points (by the Banach-Alaoglu theorem).\\
 \item \textbf{Step 2:} Identification of the limit using equidistribution theorems (only one accumulation point).
 \end{itemize}

  \subsection*{Structure of the paper} 
 
 In Section \ref{section1} we remind the reader of some standard facts about the geometry in negative curvature, Gibbs cocycles and about Gibbs measures generalizing the Bowen-Margulis-Sullivan measures on the unit tangent bundle to provide the generalization of Roblin's equidistribution theorem by Paulin, Pollicott and Schapira.  \\In Section \ref{section3} we prove fundamental estimates on the Harish-Chandra function. \\In Section \ref{section4} we prove uniform boundedness for the sequences of operators using property RD of de la Harpe et Jolissaint and the amenability of the action on the boundary, thus concluding \textbf{Step 1} of the proof of Theorem A.\\ In Section \ref{section5} we use Paulin-Pollicott-Schapira's equidistribution theorem to achieve \textbf{Step 2} of the proof of Theorem A. \\ In Section \ref{conclusion} we prove our Theorem A and its corollaries as well as Theorem D. 
\\
 
 \subsection*{Acknowledgements}
 We would like to thank Fran\c{c}ois Ledrappier and Fr\'ed\'eric Paulin for suggesting us to study Gibbs measures. We would also like to thank Uri Bader for useful discussions and criticisms. We would like to thank Vadim Kaimanovich for suggesting that we assume that $\widetilde{F}$ is symmetric. We would like also to thank Micha\l \ Zydor and Adrien Borne for their comments on this work. Finally we thank S\'ebastien Gou\"ezel for providing helpful comments about the result of Chris Connell and Roman Muchnik in \cite{CM}.
 
   \section{Preliminaries}\label{section1}
 
 \subsection{Geometry of negative curvature and potential functions}\label{CATspaces}

Recall that $X$ is a complete simply connected Riemannian manifold with dimension at least $2$ and pinched sectional curvature $-b^{2}\leq K\leq -1$ with $b\geq 1$, equipped with its  Riemannian distance denoted by $d$. The geometric boundary or the boundary at infinity, also called Gromov boundary is denoted by $\partial X$. We consider $\Gamma$ a non-elementary discrete group of isometries of $X$.\\ 
\subsubsection{Busemann functions, Bourdon's metric}
Let $x$ be in $X$, let $r$ be a geodesic ray and define the Busemann function associated with the geodesic ray $r$ as $$b_{r}(x)=\lim_{t\rightarrow \infty} d(x,r(t))-t.$$ 

Let $x$ and $y$ be in $X$ and consider the unique semi-infinite geodesic  $[xy)$ passing through $x$ and $y$, starting at $x$. Define $w^{y}_{x}$ as the unique point at the boundary so that 
\begin{equation}\label{pointbord}
w_{x}^{y}:=[xy)\cap\partial X.
 \end{equation}
The limit $\lim_{t\to \infty} d(x,r(t))-d(y,r(t))$ exists, is equal to $b_{r}(x)-b_{r}(y)$, and is independent of the choice of $r$. The horospherical distance from $x$ to $y$ relative to $v$ is defined as 
\begin{equation}\label{horospherical}
\beta_{v}(x,y)=\lim_{t\rightarrow \infty} d(x,r(t))-d(y,r(t)).
\end{equation}
Recall that the Gromov product of two points $a,b\in X$ relative to $x\in X$
is 
\[
	(a,b)_x=\frac{1}{2}(d(x,a)+d(x,b)-d(a,b)).
\]
Let $v,w$ be in $\partial X$ such that $v\neq w$. If $a_n\to v\in\partial X$, $b_n\to w\in\partial X$, then
\[
	(v,w)_x=\lim_{n\to\infty}(a_n,b_n)_x
\]
exists and does not depend on the sequences $a_n$ and $b_n$.\\ 
If $r$ is a geodesic ray representing $v$ we have: \[
	(v,y)_x=\lim_{t\rightarrow +\infty}\frac{1}{2}(d(x,r(t))+d(x,y)-d(r(t),y)),
\]
then we obtain:
\begin{equation}\label{buseman}
\beta_{v}(x,y)=2(v,y)_{x}-d(x,y).
\end{equation}

 Thus, if $z\in X$ is a point on the geodesic connecting
$v$ and $w$, then
$$
   (v,w)_x=\frac{1}{2}(\beta_v(x,z)+\beta_w(x,z)). 
$$

The geometric boundary is endowed with the Bourdon metric which defines the same topology on the boundary as the cone topology (see \cite[Chapitre III.H, Proposition 3.7 and Proposition 3.21]{BH}. Indeed the formula 
\begin{equation}\label{distance}
	d_x(v,w)={\rm e}^{-(v,w)_x}
\end{equation}
defines a metric on $\partial  X$ when we set $d_x(v,v)=0$. This is due to Bourdon and we refer to \cite[Th\'eor\`eme 2.5.1]{Bou} for more details. 
We have the following comparison formula:
\begin{equation}
	d_y(v,w)={\rm e}^{{\frac{1}{2}\left(\beta_v(x,y)+\beta_w(x,y)\right)}}d_x(v,w).
\end{equation}
 
  If $x$ and $y$ are points of $X$ and $R$ is a positive real number, we define the shadow $\mathcal{O}_{R}(x, y)$ to be the set of $v$ in $\partial X$ such that the geodesic ray issued from $x$ with limit point $v$ hits the closed ball of center $y$ with radius $R>0$.

The Sullivan shadow lemma is a very useful tool in ergodic theory of discrete groups, and it has been generalized to the context of Gibbs measure by Moshen in \cite{Mo}, see also \cite[Proposition 11.1]{Pau}.
\begin{lemma}\label{moshen}
	Let $\Gamma$ be a discrete group of isometries of $X$ and $\nu^{F}$ be a Patterson density of dimension $\sigma$ for $(\Gamma,F)$. For all $\sigma\geq \sigma_{\Gamma,F}$ and for any compact subset $K\subset X$ there exists a positive contant $C>0$ such that for all $x$ and $y$ in $\Gamma K\subset X$: $$\frac{1}{C}{\rm e}^{ {d^{F}(x,y) -\sigma d(x,y)}}\leq \nu_{x}^{F}(O_{R}(x,y))\leq C {\rm e}^{{d^{F}(x,y) -\sigma d(x,y)}}.$$
	\end{lemma}
	Assuming that $\Gamma$ is convex cocompact we will use the above lemma with $K\subset CH(\Lambda_{\Gamma})$ being the closure of a fundamental domain of the action of $\Gamma$ acting on $CH(\Lambda_{\Gamma})$. If $\Gamma$ is cocompact, then the limit set is the entire geometric boundary and the shadow lemma holds everywhere on $X$. \\

	 We say that $X$ is a $\delta$-hyperbolic space if we have the following inequality: for all     $x,y,z,t\in \overline{X}$ 
\begin{equation}\label{hyperbolic}
(x,z)_{t}\geq \min \lbrace (x,y)_{t},(y,z)_{t} \rbrace-\delta, 
\end{equation}
see \cite[3.17 Remarks (4), p. 433]{BH}. 
Using the Bourdon metric on the boundary we can compare a shadow to certain balls. More precisely we have the following proposition. This lemma, rather easy and well known, will be very useful since the boundary admits the structure of a metric space.
	\begin{lemma}\label{lemme22}
	
	\begin{enumerate}
	\item Let $R\geq 4\delta$. Then $$B(w^{y}_{x},{\rm e}^{-d(x,y)})\subset O_{R}(x,y).$$
	\item Let any $R >0$, and set $C={\rm e}^{2\delta+R}$. Then $$O_{R}(x,y)\subset B(w^{y}_{x},C{\rm e}^{-d(x,y)}).$$
	\end{enumerate}
	
	\end{lemma}
\begin{proof}
We first prove the left inclusion. Let $v$ such that $(v,w^{y}_{x})_{x}>d(x,y)$. We  Let $z$ be on $[xv)$ such that $d(x,z)=d(x,y)$.
We have $d(y,z)=d(x,y)+d(x,z)-2(y,z)_{x}$. We have 
\begin{align*}
(y,z)_{x}&\geq \min\{ (y,w^{y}_{x}),(z,w^{y}_{x})\}-\delta\\
& =\min\{ d(x,y),(z,w^{y}_{x})\}-\delta\\
& \geq\min\{ d(x,y),(z,v)_{x},(v,w^{y}_{x})\}-2\delta \\
& \geq\min\{ d(x,y),d(x,z),d(x,y)\}-2\delta \\
& =d(x,y)-2\delta, 
\end{align*}
it follows that $d(y,z)\leq 4\delta$.\\
We now prove the right inclusion. Let $v\in O_{R}(x,y)$ such that $[yv)\cap B_{X}(y,R)\neq \varnothing$ and let $z\in [yv)$ so that $d(y,z)<R$. We have 
\begin{align*}
(v,w^{y}_{x})&\geq \min\{(v,y),(y,w^{y}_{x}) \}-\delta \\
&= \min\{(v,y),d(x,y) \}-\delta \\
&\geq \min\{(v,z)_{x},(z,y)_{x},d(x,y) \}-2\delta \\
&=\min\{d(x,z),(z,y)_{x},d(x,y) \}-2\delta \\
&\geq\min\{d(x,y)-R,d(x,y)-R,d(x,y) \}-2\delta \\
&=d(x,y)-R-2\delta.
\end{align*}

\end{proof}

\subsubsection{Gibbs Cocycle and some geometric properties}
Given $\widetilde{F}: T^{1}X \rightarrow \mathbb{R}$ a $\Gamma$-invariant H\"older-continuous potential we define, as in \ref{gibbs} from the Introduction, the Gibbs cocycle $C^{F}_{v}(x,y)$ where $v\in \partial X$ and $x,y\in X$. We shall give some properties of the Gibbs cocycle but first of all note that if $\widetilde{F}=-1$ then  $$C^{F}_{v}(x,y)=\beta_{v}(x,v).$$
Hence, for every $s\in \mathbb{R}$ we have:
\begin{equation}\label{gibbsbusemann}
C_{v}^{F-s}(x,y)=C_{v}^{F}(x,y)+s\beta_{v}(x,y).
\end{equation}
Observe that if $x$ belongs to the geodesic ray from $y$ to $v$ then 
$$C_{v}^{F}(x,y)=\int_{x}^{y}\widetilde{F}.$$
The Gibbs cocycle satisfies the following cocycle property: for all $x,y,z\in X$ and $v\in \partial X$ we have 
\begin{equation}\label{cocycle}
C_{v}^{F}(x,z)=C_{v}^{F}(x,y)+C_{v}^{F}(y,z)\mbox{ and } C_{v}^{F}(y,x)=-C_{v}^{F}(x,y),
\end{equation}
and the following $\Gamma$-invariance property: for all $\gamma \in \Gamma$, all $x,y\in X$ and $v\in \partial X$: 
\begin{equation}\label{invariance}
C_{\gamma v}^{F}(\gamma x,\gamma y)=C_{v}^{F}(x,y).
\end{equation}
We now provide a lemma stating some useful properties and local estimates of the Gibbs cocycle. 
\begin{lemma}\label{fondamental}
Fix $R>0$ and assume that $\widetilde{F}$ is bounded on $p^{-1}(CH(\Lambda_{\Gamma}))\subset T^{1}X$. There exists  positive constants $C(R), D(R)$ and $E(R)$ so that:
\begin{enumerate}

\item For all $x\in CH(\Lambda_{\Gamma})$ and for all $y\in X$ such that $d(x,y)\leq R$ and for all $v\in \partial X$ we have $$|C^{F}_{v}(x,y) |\leq C(R).$$
\item For all $x$ in $X$, for all $y\in CH(\Lambda_{\Gamma})$ and for all $v\in O_{R}(x,y)$ we have $$\big|C_{v}^{F}(x,y)+d^{F}(x,y) \big|\leq D(R). $$
\item For all $x\in CH(\Lambda_{\Gamma})$ and for all $y,z$ such that $d(y,z)\leq R$ we have $$|d^{F}(x,y)-d^{F}(x,z)|\leq E(R). $$

\end{enumerate}

\end{lemma}
For a proof of Item $(1)$ and $(2)$ see \cite[Lemma 3.4]{Pau} and for Item $(3)$ see \cite[Lemma 3.2]{Pau}.

\subsection{Gibbs Cocycles, cohomology, periods and unitary representations}\label{cohomo'}

Following \cite{Led}, we recall some fundamental correspondences between potential functions, H\"olderian cocycles, and periods. We complete theses fundamental observations by adding a correspondence dealing with unitary boundary representations.\\

We say that a function defined on the boundary $\partial X$ is H\"older-continuous if it is H\"older-continuous with respect to Bourdon's metric associated with some base boint in $X$. Note that this definition does not depend on the choice of the base point. 
 We say that a cocycle $C:\Gamma \times \partial X \rightarrow \mathbb{R}$ is a \emph{H\"olderian cocycle} if for all $\gamma$ the map $C(\gamma, \cdot )$ is H\"older-continuous and if it satisfies the cocycle equality $$C(\gamma_{1}\gamma_{2}, v)= C(\gamma_{1}, \gamma_{2}v)+C(\gamma_{2},v),$$ for all $\gamma_{1},\gamma_{2}\in \Gamma$ and for all $v\in \partial X$.  
We say that two cocycles $C$ and $C'$ are \emph{cohomologous} if there exists a function $H:\partial X \rightarrow \mathbb{R}$ such that
\begin{equation}\label{cohoC}
C(\gamma,\xi)-C'(\gamma,\xi)=H(\gamma \xi)-H(\xi).
\end{equation}
Let $\gamma$ be a hyperbolic isometry, also called \emph{loxodromic} element, and denote by $\gamma^{+}$ its attractive fixed point. Observe that the quantity  $$C(\gamma,\gamma^{+}) $$
depends only on the conjugacy class of $\gamma$. 
Let  $x$ be on the axis of $\gamma$ and consider the cocycle $$C:(\gamma,v) \in\Gamma \times \partial X \mapsto C_{v}^{F}(x,\gamma x).$$
Observe that $$C(\gamma,\gamma_{+})=\int_{\gamma x}^{x}\widetilde{F}$$ and by assumption on the symmetry of $\widetilde{F}$ we have also $$C(\gamma,\gamma_{+})=\int_{ x}^{\gamma x}\widetilde{F}.$$ We call the quantity $\int_{ x}^{\gamma x}\widetilde{F}$ the period of $\gamma$ and we denote it by $\rm{Per}(\gamma)$. The set 
\begin{equation}
\rm{Per}(\widetilde{F}):=\lbrace \rm{Per}(\gamma)\mbox{ with $\gamma$ a loxodromic element} \rbrace,
\end{equation}
is called \emph{the periods of $\widetilde{F}$}. Observe that if $\widetilde{F}=1$ then $\rm{Per}(\gamma)$ is nothing else than the translation length of $\gamma$ and $\rm{Per}(\widetilde{F})$ is the length spectrum of $M$.
\begin{remark}\label{period}
 Observe also that this definition of \emph{periods of $\widetilde{F}$} coincides with the definition of \emph{periods of a H\"olderian cocycle} in \cite{Led}.
 \end{remark}
Let $\widetilde{F}^{*}=T^{1}X\rightarrow \mathbb{R}$ be another H\"older-continuous $\Gamma$-invariant function. We say that \emph{$\widetilde{F}^{*}$ is cohomologous to $\widetilde{F}$}  if there exists a function differentiable along every flow line $G: T^{1}X\rightarrow \mathbb{R}$ such that 
\begin{equation}\label{cohoF}
\widetilde{F}^{*}(v)-\widetilde{F}(v)=\frac{d}{dt}_{|t=0} G(g_{t}v ).
\end{equation}

Consider the cocycle $$C^{F}:(\gamma,v)\mapsto C_{v}^{F}(x,\gamma x).$$ First note that if $\widetilde{F}$ is bounded then $C^{F}$ is H\"older-continuous . Then observe that if $\widetilde{F}^{*}$ is cohomologous to $\widetilde{F}$  then $C^{F}$ and $C^{F^*}$ are cohomologous
(see \cite[\S 3.3 Remarks and Proposition 3.5]{Pau} for more details). The periods are an invariant of the cohomology class of potentials and also of cocycles. We have
\begin{prop}\label{equivcohomo}
Let $\widetilde{F}$ and $\widetilde{G}$ be two H\"older continuous  $\Gamma$-invariant functions on the unit tangent bundle of $X$.  Pick $x\in X$. The following assertions are equivalent 
\begin{enumerate}
\item $\nu^{F}_{x}$ and $\nu_{x}^{G}$ are in the same class.
\item The functions $\widetilde{F}$ and $\widetilde{G}$ have the same periods.
\item The Gibbs cocycles $C^{F}$ and $C^{G}$ associated with $\widetilde{F}$ and $\widetilde{G}$ are cohomologous in restriction to $\Omega \Gamma$.
\end{enumerate} 
\end{prop} 
\begin{proof}
Thanks to Remark \ref{period} we refer to \cite[\S III, Proposition 1]{Led} for the equivalence between $(1)$ and $(2)$ and between $(1)$ and $(3)$.\\ For $(2)$ implies $(3)$ we refer to \cite[Remark 3.1]{Pau}.

\end{proof}
At the level of unitary representations we  say that $\pi_{\nu_{x}^{F} }$ and  $\pi_{\nu_{x}^{G} }$ are equivalent if there exists a unitary operator $
U:L^{2}(\partial X, \nu^{F}_{x})\rightarrow L^{2}(\partial X, \nu^{G}_{x})$ such that: 
\begin{equation}\label{cohoR}
U\pi_{\nu_{x}^{F} }=\pi_{\nu_{x}^{G} }U.
\end{equation}

\begin{lemma}\label{cohomo}
Let $\nu^{F}$ and $\nu^{F^*}$ be two Patterson densities of dimension $\sigma$. Pick $x$ in $X$ and consider the unitary representations $\pi_{\nu_{x}^{F} }$ and $\pi_{\nu_{x}^{F^*} }$. If $F$ and $F^*$ are cohomologous then $\pi_{\nu_{x}^{F} }$ and  $\pi_{\nu_{x}^{F^*} }$ are equivalent.
\end{lemma}
\begin{proof}
Since $\widetilde{F}$ and $\widetilde{F}^{*}$ are cohomologous then the cocycles $C^{F}:(\gamma,v)\mapsto C_{v}^{F}(x,\gamma x)$ and $C^{F^{*}}:(\gamma,v)\mapsto C_{v}^{F^{*}}(x,\gamma x)$ are cohomologous. Thus the multiplication operator by ${\rm e}^{\frac{1}{2}H}$ from  $L^{2}(\partial X,\nu^{F}_{x})$ to $L^{2}(\partial X,\nu^{F^{*}}_{x})$ intertwines the unitary representations $\pi_{\nu_{x}^{F} }$ and $\pi_{\nu_{x}^{F^{*}} }$ where $H:\partial X \to \mathbb{R}$ satisfies the identity (\ref{cohoC}).
\end{proof}

\subsection{Gibbs measures and Roblin-Paulin-Pollicott-Schapira's equidistribution theorem}\label{Gibbs}
\subsubsection{Hopf parametrization}
Let us now recall  a parametrization of $T^{1}X$ in terms of the boundary at infinity of $X$. \\
If $\rm{v}$ $=(x,\vec{ v })$ is an element of $T^{1}X$, consider the unique geodesic defined by $\vec{ v }$ represented by an isometry $r:\mathbb{R}\rightarrow X$ such that $r(0)=q({ \rm v})$ and $\frac{d}{dt}_{|t=0}r=\vec{ v }$. We denote by $v_{-}$ and $v_{+}$ the endpoints of the geodesic such that $r(-\infty)=v_{-}$ and  $r(+\infty)=v_{+}$.\\
Let us define $\partial^{2}X=\partial X \times \partial X-\Delta$, where $\Delta$ is the diagonal of $\partial X\times \partial X$. For every base point $x_{0}$ in $X$, the space $T^{1}X$ may be identified with $\partial^{2} X\times \mathbb{R}$, by the map which maps a unit tangent vector $ {\rm v} $ to the triple $(v_{-},v_{+},t)$ where $t$ represents the algebraic distance on the image of the geodesic represented by $r$ between $r(0)$ and the closest point of the geodesic to $x_{0}$. This parametrization, depending a priori on $x_{0}$, differs from the one defined by another base point $x_{0}'$ only by an additive term on the third factor (independent of the time $t$).\\

\subsubsection{The potential gap}
For all $x$ in $X$ and for all $v,w\in \partial X$ define the \emph{gap map} as $$D_{x,F}(v,w):=\exp{\bigg( \frac{1}{2}\bigg(\int_{x}^{w_{t}}\widetilde{F} -\int_{v_{t}}^{w_{t}}\widetilde{F}+ \int_{v_{t}}^{x}} \widetilde{F} \bigg) \bigg).$$
We observe that $D_{x,F}$ generalizes Bourdon's metric $d_{x}$ since for $F=-1$ we obtain $D_{x,F}=d_{x}$. Note the $\Gamma$ invariance property $D_{\gamma x, F}(\gamma v,\gamma w)=D_{x,F}(v,w)$ for all $\gamma \in \Gamma$ and for all $v,w\in \partial X$.
\subsubsection{The Gibbs states of $(\Gamma,F)$}
Let $\sigma$ be a real number and let $(\nu^{F}_{x})_{x\in X}$ be a Patterson density of dimension $\sigma$ for $(\Gamma,F)$. Once we have fixed a base point $x_{0}\in X$ and used the Hopf parametrization, define \emph{the Gibbs measures on $T^{1}X$ associated with $(\nu^{F}_{x})_{x\in X}$} as   
\begin{equation}\label{Gibbsmeasure}
dm(v)=\frac{d\nu_{x_{0} }(v_{-})  d\nu_{x_{0} }(v_{+}) dt }{D^{2}_{F-\sigma,x_{0}}(v_{-},v_{+})}.
\end{equation}
The groups $\Gamma$ and $\mathbb{R}$ act on $\partial^{2} X\times \mathbb{R}$ via $\gamma(v_{-},v_{+},t)=(\gamma v_{-}, \gamma v_{+},t)$ and via the goedesic flow $s(v_{-},v_{+},t)=(v_{-},v_{+},t+s)$. Observe that both actions commute. Thus define $m_{F}$ on $\Gamma \backslash T^{1}X=T^{1}M$, and we call $m_{F}$ \emph{the Gibbs measures on $T^{1}M$ associated with $(\nu^{F}_{x})_{x\in X}$}. If $\|m_{F}\|<\infty$ we say that $m_{F}$ is finite. The finiteness of the Gibbs measures will  always be satisfied when we consider convex cocompact groups.

\subsubsection{Mixing property of Gibbs measures}

 We say that $g_{t}$ is mixing on $\Gamma\backslash T^{1}X$ with respect to $m_{F}$ if for all bounded Borel subsets $A,B \subset \Gamma\backslash T^{1}X$ we have $\lim_{t \rightarrow +\infty} m_{F}(A\cap g_{t}(B))=m_{F}(A)m_{F}(B)$. 

There exist a condition which guarantees that the geodesic flow on $T^{1}X$ is mixing and it is related to the non-arithmeticity of the spectrum of $\Gamma$ and this is due to Babillot. We refer to \cite[Proposition 7.7]{Bab} for a proof of this fact in the case of Patterson-Sullivan measures and to \cite[Theorem 8.1]{Pau} for Patterson densities. 

We have finished the preparations to state a theorem and one of its corollaries which will be one of our main tools. The main idea of these equidistribution results goes back to the pioneering work of Margulis \cite{M} who made a connection between the mixing property of the geodesic flow with the counting of closed geodesics on a compact negatively curved manifold. The form of the following equidistribution results, due to Paulin, Pollicott and Schapira \cite[Theorem 9.1]{Pau}, is inspired by the results of T. Roblin in \cite[Th\'eor\`eme 4.1.1]{Ro}.
 \begin{theorem}\label{roblin}(Paulin, Pollicott and Schapira)
Let $\Gamma$ be a discrete group of isometries of $X$ and assume that $\sigma_{\Gamma,F}$ is finite and positive. Assume that $m_{F}$ is finite and mixing under the geodesic flow on $T^{1}M$. Then for all $x,y \in X$ and for all $c>0$: $$\frac{\sigma_{\Gamma,F} \|m_{F}\|}{1-{\rm e}^{-c\sigma_{\Gamma,F}}} {\rm e}^{-\sigma_{\Gamma,F} n}\sum_{ \left\{\gamma \in \Gamma| n-c<d( x,\gamma  y)\leq n  \right\}} {\rm{e}}^{d^{F}(x,\gamma y)}D_{\gamma^{-1}  x} \otimes D_{\gamma  y} \rightharpoonup \nu^{F}_{x} \otimes \nu^{F}_{y} $$ as $n\rightarrow +\infty$  with respect to the weak* convergence of $C(\overline{X} \times \overline{X})^{*}$.
\end{theorem}

As a corollary we obtain the following result that we shall use in the Step 2 of the computation of the limit in Section \ref{section5}.\\
For a subset $A$ in $\partial X$ with a vertex $x$, denote by $\mathcal{C}_{x}(A)$ that is the union of the geodesic rays or lines starting from $x$ and ending at $A$, and this a subset of $\mathcal{C}_{x}(A)\subset \overline{X}$ so that $\mathcal{C}_{x}(A)\cap \partial X=A$. 

\begin{coro}\label{coroPaulin}
Let $\Gamma$ be a discrete group of isometries of $X$ with a non-arithmetic spectrum. Assume that $m_{F}$ is finite and mixing under the geodesic flow on $T^{1}M$. If $U$ and $V$ are two Borel sets, then for all $x,y \in X$ and for all $c>0$: $$\limsup_{n \rightarrow +\infty} \frac{\sigma_{\Gamma,F} \|m_{F}\|}{1-{\rm e}^{-c\sigma_{\Gamma,F}}} {\rm e}^{-\sigma_{\Gamma,F} n}\sum_{C_{n}(x)} {\rm{e}}^{d^{F}(x,\gamma x)} (D_{\gamma x} \otimes D_{\gamma^{-1}x})(\chi_{\mathcal{C}_{x}(U)} \otimes \chi_{\mathcal{C}_{x}(V)})\leq \nu^{F}_{x}(\overline{U})  \nu^{F}_{x} (\overline{V}).$$ 

\end{coro}

We recall that we have defined the normalization constant $c_{\Gamma,F}$ as $$c_{\Gamma,F}:=  \frac{\sigma_{\Gamma,F} \|m_{F}\|}{1-{\rm e}^{-\sigma_{\Gamma,F}}} .$$

\section{The Harish-Chandra function}\label{section3}

The goal of this section is to prove the following estimate on the Harish-Chandra function.

\begin{prop}(Harish-Chandra's estimate)\label{HCH} Let $\nu^{F}=(\nu^{F}_{x})_{x\in X}$ be a Patterson density of dimension $\sigma_{\Gamma,F}$. There exists a constant $R>0$ and a constant $C > 0$ (depending on $R$) such that for all $\gamma \in \Gamma$  satisfying $d(x,\gamma x)\geq R$ with $x$ in $CH(\Lambda_{\Gamma})$ we have 

\begin{center}

$C^{-1}d(x,\gamma x) {\rm e}^{{\frac{1}{2} d^{F}(x,\gamma x) -\frac{1}{2}\sigma_{\Gamma,F} d(x,\gamma x) }} \leq \phi_{x}(\gamma ) \leq Cd(x,\gamma x)  {\rm e}^{{\frac{1}{2} d^{F}(x,\gamma x) -\frac{1}{2}\sigma_{\Gamma,F} d(x,\gamma x) }}.$ 

\end{center}

\end{prop}

\begin{remark}
It would be probably more appropriate to call these estimates \emph{Harish-Chandra Anker's estimates} because Anker has improved estimates established by Harish-Chandra in the setting of semisimple Lie groups. He improved notably  the lower bound by adding a polynomial, see \cite{An}. 
\end{remark}
\subsection{Some technical lemmas}
The following lemma is due to S. Alvarez in \cite{Alv}. Since our methods are rather analytical and since our conventions are different, we give another shorter proof.

\begin{lemma}\label{alv1}
There exits $r>0$ such that for all $x,y,z$ aligned in this order we have for all $v\in\partial X \backslash O_{R}(x,y)$:
$$\beta_{v}(y,z)\leq r-d(y,z).$$
\end{lemma}

\begin{proof}
We have $\beta_{v}(y,z)=2(v,z)_{y}-d(y,z)$. The hyperbolic inequality (\ref{hyperbolic}) implies that $(v,z)_{y}\leq (v,w^{z}_{y})_{y}+\delta$. An upper bound of the quantity $ (v,w^{z}_{y})_{y}$ is equivalent to a lower bound of $d_{y}(v,w_{y}^{z})$.
We have $\beta_{w_{y}^{z}}(x,y)=d(x,y)$ (because $w_{y}^{z}=w_{x}^{y}$) and $\beta_{v}(x,y)=2(v,y)_{x}-d(x,y) $. Thus the hyperbolic inequality (\ref{hyperbolic}) implies
\begin{align*}
\beta_{v}(x,y)  +\beta_{w_{y}^{z}}(x,y)=2(v,y)_{x}\geq &2\min\{(v,w_{x}^{y}),(w_{x}^{y},y)\}-2\delta\\
\geq& 2\min\{(v,w_{x}^{y}),d(x,y)\}-2\delta.
\end{align*} 
Since $v\in\partial X \backslash O_{R}(x,y)$ we have $d(v,w_{x}^{y})\geq {\rm e}^{-d(x,y)}$ by Lemma \ref{lemme22}, equivalently $(v,w_{x}^{y})\leq d(x,y)$.
Thus $${\rm e}^{{\frac{1}{2}(\beta_{v}(x,y)  +\beta_{w_{y}^{z}}(x,y) )}}\geq \frac{{\rm e}^{2\delta}}{d_x(v,w_{x}^{y})}=\frac{{\rm e}^{2\delta}}{d_x(v,w_{y}^{z})}.$$
By the conformal equivalence of the metric on the boundary we have: $$d_{y}(v,w_{y}^{z})=d_{x}(v,w_{y}^{z}) {\rm e}^{{\frac{1}{2}(\beta_{v}(x,y)  +\beta_{w_{y}^{z}}(x,y) )}},$$
hence $${\rm e}^{2\delta} \leq d_{y}(v,w_{y}^{z}),$$ 
we set  $r={\rm e}^{2\delta}$ to conclude the proof.
\end{proof}

Before proceeding we will need to set up some notation. We follow the decomposition used by Alvarez in \cite{Alv}.\\ 
\subsection*{Definition of $A_{i,R}(\gamma )$.} Fix $R>0$ such that Lemma \ref{lemme22} is available.
Let $\gamma$ be in $\Gamma$ such that $d(x,\gamma x)\geq R$  and consider the geodesic $[xw^{\gamma x}_{x})$ starting at $x$ and passing through $\gamma x$ and ending at $w^{\gamma x}_{x} \in \partial X$.
Let $z_{i}$ for $i=0,\cdots, N$  be a finite sequence of points belonging to $[xw_{x}^{\gamma x})$ aligned in the following order: $z_{N},\cdots, z_{0} $,
 with $z_{0}=\gamma x$ and so that the choice of $z_{N}$ satisfies  $d(x,z_{N})<R/2$,  
 and $d(z_{i},z_{i+1})=R/2$. Observe that $d(x,\gamma x)=d(x,z_{N})+N\frac{R}{2}$.\\
 For $i=1,\cdots,N$  we set $$A_{i,R}(\gamma):=O_{R}(x,z_{i}) \backslash O_{R}(x,z_{i-1}), $$
 Observe that $A_{N,R}=\partial X\backslash O_{R}(x,z_{N-1})$.

 We can decompose the boundary as the following disjoint union 
\begin{equation}\label{decompo}
\partial X:= \sqcup^{N}_{i=1} A_{i,R}(\gamma)\sqcup O_{R}(x,\gamma x).
\end{equation}

\begin{prop}\label{alvarez}
We suppose here that $\widetilde{F}$ is symmetric. Let $\nu^{F}$ be a Patterson density of dimension $\sigma\geq \sigma_{\Gamma,F}$ so that the estimates in Mohsen's Shadow lemma hold (Lemma \ref{moshen}).
There exists a constant $C > 0$ such that for all $\gamma \in \Gamma$, $v\in {A}_{i,R}(\gamma)$, and $1 \leq i \leq N$ we have that 

\begin{center}

$C^{-1}{\rm e}^{{d^{F}(z_i,\gamma x)-\sigma d(z_{i},\gamma x)}} \leq {\rm e}^{{  C^{F-\sigma}_{v} (z_{i},\gamma x)}} \leq C{\rm e}^{{{d^{F}(z_i,\gamma x)-\sigma d(z_{i},\gamma x)}}}. $

\end{center}
\end{prop}

\begin{proof}
Recall that $C_{v}^{F-\sigma}(z_{i},\gamma x)=C_{v}^{F}(z_{i},\gamma x)+\sigma \beta_{v}(z_{i},\gamma x)$.

If $\beta_{v}(z_{i},\gamma x)\leq 0$ then Lemma \ref{alv1} implies that $d(z_{i},\gamma x)\leq r$ for some positive real number $r$. Therefore the estimates follow from Lemma \ref{fondamental} Item (2).\\
Now we call $z'$ the point of the intersection of the horosphere centered at $v$ passing through $\gamma x$ and the geodesic passing through $v$ and $\gamma x$. If $\beta_{v}(z_{i},\gamma x)<0$,  then $\gamma x,z'$ and $v$ are aligned in this order. Thus we can write $$\lim_{t \rightarrow +\infty}\int_{\gamma x}^{v_{t}}F=\int_{\gamma x}^{z'}F+\lim_{t \rightarrow +\infty} \int_{z'}^{v_{t}}F.$$
Since $\beta_{v}(\gamma x,z')=d(\gamma x,z')$ we have $$C^{F-\sigma}_{v}(z_{i},\gamma x)= C^{F-\sigma}_{v}(z_{i},z')+\int_{\gamma x}^{z'} F-\sigma,$$
besides, the symmetry of $F$ implies 
\begin{equation}\label{revers}
C^{F-\sigma}_{v}(z_{i},\gamma x)=C^{F-\sigma}_{v}(z_{i},z')+\big(d^{F}(z',\gamma x)-\sigma d(z',\gamma x)\big). 
\end{equation}
Notice that
\begin{align*}
d(z_{i},z') &\leq d(z_{i},\gamma x)+d(\gamma x,z')\\
&=d(z_{i},\gamma x)+\beta_{v}(\gamma x,z')\\
&=d(z_{i},\gamma x)+\beta_{v}(\gamma x,z_{i})\\ 
&\leq d(z_{i},\gamma x)+r-d(\gamma x,z_{i})\\
&=r.
\end{align*}
Thus, the fist term on the right hand side equality (\ref{revers}) is bounded by Lemma \ref{fondamental}, Item (1). In the second term on the right hand side equality (\ref{revers}) the quantity $d^{F}(z',\gamma x)\leq d^{F}(z_{i},\gamma x)+E(r)$, for some positive constant $E(r)$ by Lemma \ref{fondamental}, Item (3); and the triangle inequality implying $d(z',\gamma x)\leq d(z_{i},\gamma x)+r$ completes the proof.

\end{proof}

\begin{prop}\label{key'}
There exists a positive constant $C>0$ such that for all $i=1,\cdots, N$ and for all $v\in A_{i,R}(\gamma)$ we have:
$$ C^{-1}{\rm e}^{{-\frac{1}{2}\sigma d(x,z_{i})}}{\rm e}^{{\sigma d(x,z_{i})}} \leq {\rm e}^{{\frac{\sigma}{2} \beta_{v}(x,z_{i})}}\leq C  {\rm e}^{{-\frac{1}{2}\sigma d(x,z_{i})}}{\rm e}^{{\sigma d(x,z_{i})}}.$$

\end{prop}

\begin{proof} 
The proof is based on the hyperbolic inequality (\ref{hyperbolic}). \\

Let us prove the right hand side inequality. We have $\beta_{v}(x,z_{i})=2(v,z_{i})_{x}-d(x,z_{i})$. We shall just control the Gromov product $(v,z_{i})_{x}$ for $v\in A_{i,R}$.

 For all $i$, for all $v$ we have: 
\begin{align*}
(v,w^{\gamma x}_{x})&\geq \min\{ (v,z_{i})_{x}, (z_{i},w^{\gamma x}_{x})_{x}\}-\delta\\
&= \min\{ (v,z_{i})_{x}, d(x,z_{i}) \}-\delta\\
&=(v,z_{i})_{x}-\delta.
\end{align*}
Therefore, 
$${\rm e}^{{\frac{\sigma}{2} \beta_{v}(x,z_{i})}}\leq {\rm e}^{\delta} \frac{{\rm e}^{{-\frac{1}{2}\sigma d(x,z_{i})}}}{d^{\sigma}_{x}(v,w^{\gamma x}_{x})}.$$
If $v\in A_{i,R}$ then $v$ is not in $O_{R}(x,z_{i-1})$, and thus $v$ is not in $B(w^{\gamma x}_{x},{\rm e}^{d(x,z_{i-1})}) $ hence by Lemma \ref{lemme22} we have $d_{x}(v,w^{\gamma}_{x})\geq {\rm e}^{-d(x,z_{i-1})}={\rm e}^{-d(x,z_{i})+\frac{R}{2}}$. We deduce 
$$
{\rm e}^{{\frac{\sigma}{2} \beta_{v}(x,z_{i})}}\leq {\rm e}^{{\sigma(\delta- \frac{R}{2})}}   {\rm e}^{{-\frac{1}{2}\sigma d(x,z_{i})}}{\rm e}^{{\sigma d(x,z_{i})}}.
$$

We prove now the left hand side inequality. 
We have \begin{align*}
(v,z_{i})_{x}&\geq \min\{ (v,w^{\gamma x}_{x})_{x}, (w^{\gamma x}_{x},z_{i})_{x}\}-\delta\\
&= \min\{ (v,w^{\gamma x}_{x})_{x}, d(x,z_{i}) \}-\delta.
\end{align*}
If $v$ is in $A_{i,R}$, then $v$ is not in $B(w^{\gamma x}_{x}, {\rm e}^{-d(x,z_{i-1})} )$. Hence $(v,w^{\gamma x}_{x})_{x}\leq d(x,z_{i-1})=d(x,z_{i})-R/2$. It follows that  
$$(v,z_{i})_{x}\geq d(x,z_{i})-R/2-\delta.$$
We deduce that $$ 
{\rm e}^{{\frac{\sigma}{2} \beta_{v}(x,z_{i})}}\geq {\rm e}^{-\sigma(\delta+ \frac{R}{2})}  {\rm e}^{{-\frac{1}{2}\sigma d(x,z_{i})}}{\rm e}^{{\sigma d(x,z_{i})}}.
$$

Hence, we set $C= {\rm e}^{\sigma(\delta+ \frac{R}{2})} $ to conclude the proof.
\end{proof}

\subsection{Proof of estimates}
 We are ready to establish the Harish-Chandra estimates:
 \begin{proof} We only prove the upper bound, the lower bound follows the same method.\\
 Pick $x\in CH(\Lambda_{\Gamma})$ and write the Harish-Chandra function as a sum of integrals over the partition (\ref{decompo}) as follows:

$$\phi_{x}(\gamma) =\sum_{i=1}^{N } \int_{A_{i,R}} {\rm e}^{{\frac{1}{2} C_{v}^{F-\sigma}(x,\gamma x) }} d\nu_{x}^{F}(v) +\int_{O_{R}(x,\gamma x)}{\rm e}^{{\frac{1}{2} C_{v}^{F-\sigma}(x,\gamma x) }} d\nu_{x}^{F}(v).$$

To prove the proposition we will show that each integral is comparable to ${\rm e}^{{\frac{1}{2}d^{F}(x,\gamma x) - \frac{1}{2} \sigma d(x, \gamma x)}}$.

\textbf{The upper bound over $O_{R}(x,\gamma x)$}: 
\begin{align*}
\int_{O_{R}(x,\gamma x)} {\rm e}^{{\frac{1}{2} C_{v}^{F-\sigma}(x,\gamma x) }}d\nu_{x}^{F}(v)&\leq  C\nu_{x}^{F}(O_{R}(x,\gamma x)){\rm e}^{{-\frac{1}{2}d^{F}(x,\gamma x) +\frac{1}{2}\sigma d(x,\gamma x)}}\\
&\leq C{\rm e}^{{\frac{1}{2}d^{F}(x,\gamma x) -\frac{1}{2}\sigma d(x,\gamma x)}},
\end{align*}

where the first inequality follows from Lemma \ref{fondamental} Item (2) since $\gamma x$ is in $CH(\Lambda_{\Gamma})$, and the second inequality follows from the upper bound of Mohsen's shadow Lemma (Lemma \ref{moshen}), the compact $K$ being the closure of a fundamental domain of the action of $\Gamma$ on $CH(\Lambda_{\Gamma})$.

\textbf{The upper bound over $A_{i,R}(\gamma)$:}
We have established two useful inequalities dealing with the terms we shall control:
the first one follows from Proposition \ref{alvarez}. There exists $C>0$ such that we have for all $i$, for all $v\in A_{i,R}(\gamma)$:
\begin{equation}\label{key1}
{\rm e}^{{\frac{1}{2}C^{F-\sigma}_{v} (z_{i},\gamma x)}}\leq C {\rm e}^{{\frac{1}{2}d^{F} (z_{i},\gamma x)-\frac{1}{2}\sigma d(z_{i},\gamma x)}}.
\end{equation}

The second one is from Proposition \ref{key'}. There exists $C>0$ so that 
\begin{equation}\label{key2}
{\rm e}^{{\frac{\sigma}{2} \beta_{v}(x,z_{i})}}\leq C {\rm e}^{{-\frac{1}{2}\sigma d(x,z_{i})}}{\rm e}^{{\sigma d(x,z_{i})}}.
\end{equation}

 Combining these two estimates will yield the bound over $A_{i,R}(\gamma)$. We will use a constant $C$ which absorbs the other constants. Now estimating over $A_{i,R}(\gamma)$ we get,

\begin{align*}
\int_{A_{i,R}(\gamma)}{\rm e}^{{\frac{1}{2} C_{v}^{F-\sigma}(x,\gamma x) }} {\rm d}\nu_{x}^{F}(v)&= \int_{A_{i,R}(\gamma)}{\rm e}^{{\frac{1}{2} C_{v}^{F-\sigma}(z_{i},\gamma x)}}  {\rm e}^{{\frac{1}{2}C_{v}^{F}(x,z_{i})}} {\rm e}^{{\frac{1}{2}\sigma \beta_{v}(x,z_{i})}}  d\nu_{x}^{F}(v)\\
\mbox{Inequality (\ref{key1})}&\leq C {\rm e}^{{\frac{1}{2}\big(d^{F}(z_{i},\gamma x)-\sigma d(z_{i},\gamma x)}\big)}  \int_{A_{i,R}(\gamma)}{\rm e}^{{\frac{1}{2}C_{v}^{F}(x,z_{i})}}{\rm e}^{{\frac{1}{2}\sigma \beta_{v}(x,z_i)}} d\nu_{x}^{F}(v)\\
 \mbox{Lemma \ref{fondamental} Item (2)}&\leq C {\rm e}^{{\frac{1}{2}\big(d^{F}(z_{i},\gamma x)-\sigma d(z_{i},\gamma x)}\big)} {\rm e}^{{-\frac{1}{2}d^{F}(x,z_{i})}}  \int_{A_{i,R}(\gamma)} {\rm e}^{{\frac{1}{2}\sigma \beta_{v}(x,z_{i})}} d\nu_{x}^{F}(v)\\
 \mbox{Inequality (\ref{key2})} &\leq C {\rm e}^{{\frac{1}{2}\big(d^{F}(z_{i},\gamma x)-\sigma d(z_{i},\gamma x)}\big)} {\rm e}^{{-\frac{1}{2}d^{F}(x,z_{i})}}    {\rm e}^{{-\frac{1}{2}\sigma d(x,z_{i})}}{\rm e}^{{\sigma d(x,z_{i})}}\nu_{x}^{F}(A_{i,R}(\gamma))\\
 &=C{\rm e}^{{\frac{1}{2}d^{F}(z_{i},\gamma x)}}{\rm e}^{{-\frac{1}{2}\sigma d(x,\gamma x)}} {\rm e}^{{-\frac{1}{2}d^{F}(x,z_{i})}}   {\rm e}^{{\sigma d(x,z_{i})}}\nu_{x}^{F}(A_{i,R}(\gamma))\\
\mbox{Moshen's shadow Lemma}&\leq  C {\rm e}^{{\frac{1}{2}d^{F}(z_{i},\gamma x)-\frac{1}{2}\sigma d(x,\gamma x)}} {\rm e}^{{-\frac{1}{2}d^{F}(x,z_{i})}}   {\rm e}^{{\sigma d(x,z_{i})}} {\rm e}^{{d^{F}(x,z_{i})-\sigma d(x,z_{i})}}\\
&=C{\rm e}^{{\frac{1}{2}d^{F}(x,\gamma x) -\frac{\sigma}{2}d(x,\gamma x)}}.
 \end{align*}

Combining the \textbf{Upper bound over $O_{R}(x,\gamma x)$} with the \textbf{Upper bound over $A_{i,R}(\gamma)$} for all $i=1,\cdots ,N$ leads to $$\phi_{x}(\gamma)\leq C(N+1) {\rm e}^{{\frac{1}{2}d^{F}(x,\gamma x) -\frac{\sigma}{2}d(x,\gamma x)}}.$$
Since $N \frac{R}{2}=d(z_{N},\gamma x)\leq d(x,\gamma x)$ we obtain the left hand side inequality of Harish-Chandra's estimates.

\end{proof}

\begin{remark}\label{Aiestim}
In particular we prove that there exists $C>0$ such that for all $i=1,\cdots,N$:\begin{equation}\label{Ai}
\int_{A_{i,R}(\gamma)}{\rm e}^{{\frac{1}{2} C_{v}^{F-\sigma}(x,\gamma x) }} d\nu_{x}^{F}(v)\leq C {\rm e}^{{\frac{1}{2}d^{F}(x,\gamma x)-\frac{1}{2}\sigma d(x,\gamma x)}}.
\end{equation}
\end{remark}

\section{Uniform boundedness via RD}\label{section4}
\subsection{Quasi-regular representations}
Let $\Gamma$ be a discrete countable group acting on a measure space $(S,\nu)$ with a $\Gamma$-quasi-invariant measure $\nu$. This action gives rise to a unitary representation after correction by the square root of the Radon Nikodym derivative of the action:
\begin{equation*}
\pi_{\nu}:\Gamma \rightarrow \mathcal{U}(L^{2}(S,\nu))
\end{equation*}
defined for $\xi \in U(L^{2}(S,\nu))$ and for $s\in S$ as 
\begin{equation}
\big(\pi_{\nu}(\gamma)\xi\big)(s)= \bigg(\frac{d\gamma_{*}\nu}{d\nu}\bigg)^{\frac{1}{2}}(s)\xi(\gamma^{-1}s).
\end{equation}
This unitary representation is called the quasi-regular representation associated with $\Gamma \curvearrowright (S,\nu)$ (also called \emph{Koopman representation}).\\

 In the following we will denote by $\lambda_{\Gamma}:\Gamma\rightarrow \mathcal{U}(\ell^{2}(\Gamma))$ the left regular representation.\\ 

 Recall that a unitary representation $\pi$  is \emph{weakly contained} in a unitary representation $\rho$ if for all functions $f\in\ell ^{1}(\Gamma)$ we have
 \begin{equation}
 \| \pi(f)\|\leq \| \rho (f)\|.
 \end{equation}
 We refer to \cite[Appendix F]{BHV} and to \cite[Section 18]{Dix} for more details.\\
Let $\rho $ be a unitary representation of $\Gamma$
and let $\mu $ be a bounded measure on $\Gamma$ and define the operator $\rho(\mu)$ as: $$\rho(\mu):=\sum_{\gamma \in \Gamma}\mu(\gamma)\rho(\gamma),$$ and observe $\rho(\mu)\in \mathcal{B}(L^{2}(S,\nu)).$ 
\subsection{Spectral characterization of amenable action}
It is well known that an amenable discrete group can be characterized by the fact that the trivial representation is weakly contained in the left regular representation.
  Kuhn was probably inspired by this property to prove an analog result for quasi-regular representations associated with ergodic \emph{amenable actions in Zimmer's sense} in \cite{K}. We describe briefly which notion of amenable action we shall consider.\\
We know since Spatzier in \cite{Spa} that the action of $\pi_{1}(M)$, the fundamental group of a compact manifold $M$, on the geometric boundary of the universal cover of $M$ is amenable in Zimmer's sense with respect to the standard measure class. Eventually, Spatzier and Zimmer showed in \cite[Theorem 3.1]{SZ} that this action is amenable with respect to \emph{any quasi-invariant measure}. Later, after the work of Adams \cite{A1}, Kaimanovich \cite{Ka} proved that the action of a closed subgroup of isometries of a hyperbolic space with a finite critical exponent ( -critical exponent- in the usual sense without a potential function) is \emph{topologically amenable}. In this paper, we consider the action of a discrete group of isometries on the geometric boundary as a topological space. The notion of topological amenability is the more appropriate notion we shall consider since the space appears naturally as a topological space  rather than only as a measurable space.
\begin{defi}
An action  $\Gamma\curvearrowright S$ on a topological space $S$ is topologically amenable if there exists a sequence of continuous maps $$\mu^{n}:s\in S \mapsto \mu^{n}_{s}\in {\rm Prob}(\Gamma)$$ of probabilities on $\Gamma$ such that $$\lim_{n \to +\infty}\sup_{s\in S} \|\gamma_{*}\mu^{n}_{s}-\mu^{n}_{\gamma s} \|\to 0 $$ as $n\to \infty$.
\end{defi}
 It turns out that in the case of a topological space \emph{topologically amenable} and \emph{amenable in Zimmer's sense} are equivalent, see \cite{ADR}. Therefore we will not have to pay attention to any quasi-invariant measure on the geometric boundary.\\ It is shown in \cite{Adams} that for a locally compact group $G$ acting on $(S,\mu)$ that the definition of amenable action in Zimmer's sense is equivalent to the existence of a $G$-equivariant conditional expectation from $L^{\infty}(G\times S)$ to $L^{\infty}(S)$. Hence if $\Gamma$ is a discrete group of isometries of a complete simply connected pinched negatively curved Riemannian manifold $X$, with a finite critical exponent, we have that $(\Gamma \times \partial X, \partial X)$ is a \emph{$\Gamma$-pair} in the sense of \cite{Ana}.

We deduce from \cite[Corollary 3.2.2]{Ana}  that
\begin{prop}\label{amenable}
 Let $\Gamma$ be a discrete group of isometries of $X$ a complete simply connected Riemmanian manifold with pinched curvature, with a finite critical exponent. For any quasi-invariant measures $\nu$ on the geometric boundary $\partial X$ we have for any bounded $\mu$ measure on $\Gamma$  $$\|\pi_{\nu}(\mu)\|\leq \|\lambda_{\Gamma}(\mu)\|.$$
\end{prop}

\begin{remark}
Indeed, by \cite[Lemma 2.3]{S} due to Shalom with the same hypothesis we have the other inequality and thus we obtain for any bounded $\mu$ measure on $\Gamma$ an equality $$\|\pi_{\nu}(\mu)\|= \|\lambda_{\Gamma}(\mu)\|. $$
\end{remark}
\subsection{Property RD}
The property RD comes from the theory of C*-algebras and has been introduced in the important paper \cite{H} by Haagerup.\\ A length function $|\cdot|$ on a discrete countable group $\Gamma$ is a function $|\cdot|: \Gamma \rightarrow \mathbb{R}^{+}$, satisfying $|e|=0$ where $e$ is the neutral element of $\Gamma$, $|\gamma^{-1}|=|\gamma|$ and $|\gamma_{1}\gamma_{2}|\leq |\gamma_{1}|+|\gamma_{2}|$. Let $s>0$ and define the Sobolev space associated with $\Gamma$ denoted by $H^{s}(\Gamma)$ as the space   $$H^{s}(\Gamma):=\big\lbrace f:\Gamma \rightarrow \mathbb{C} \mbox{ such that } \|f\|^{2}_{H^{s}}:= \sum_{\Gamma} |f(\gamma)|^{2}(1+|\gamma|)^{2s}<\infty \big\rbrace .$$ \\ Given a discrete countable group equipped with a length function $|\cdot|$ we say that \emph{$\Gamma$ satisfies property RD with respect to $|\cdot|$} if the space  $H^{s}$ convolves $H^{s}(\Gamma)*\ell^{2}(\Gamma)\subset \ell^{2}(\Gamma)$ in the following way: $$\mbox{  $\exists C,s>0$ such that  for all $f\in H^{s}(\Gamma),\xi \in \ell^{2}(\Gamma)$, we have $\|f*\xi\|_{2}\leq C \|f\|_{H^{s}}\|\xi\|_{2}$}.$$
 In terms of operator norm, property RD means that there exist two positive constants $C$ and $s>0$ such that the multiplication operator by convolution by a function in $H^{s}(\Gamma)$ is continuous: 
$$\exists C,s>0 \mbox{ such that for all $f\in H^{s}(\Gamma)$ we have } \|\lambda_{\Gamma}(f)\| \leq C\|f\|_{H^{s}}.$$ 
This inequality means, in operator algebraic terms that we have the continuous inclusion 
\begin{equation}
H^{s}(\Gamma)\hookrightarrow
C_{r}^{*}(\Gamma).
\end{equation}
\begin{remark}
If we specialize the property RD to the abelian group $\mathbb{Z}$ (with its standard word length function) we obtain the well known fact, using the Fourier transform, that an element in $L^{2}(\mathbb{S}^{1})$ with Fourier coefficients ``Rapidly Decreasing" to $0$ define a continous function on the circle $\mathbb{S}^{1}$.
\end{remark}
We extract the following inequality expressed in norm of convolution operators established in \cite{Jo} to prove the property RD for convex cocompact groups. 

\begin{prop}\label{RD}
Let $X$ be a complete simply connected Riemannian pinched negatively curved manifold and
let $\Gamma$ be a convex cocompact discrete group of isometries of $X$. Pick a point $x$ in $X$ and recall the definition of an annulus $C_{n}:=C_{n}(x)$. Let $\chi_{n}$ be the characteristic function of $C_{n}$, then $$\|\lambda_{\Gamma}(f\chi_{n})\|\leq Cn\|f\|_{2}.$$
\end{prop} 
\begin{proof}
In \cite[Proposition 3.2.4]{Jo}, Jolissaint proves that there exists a positive constant $c$, depending only on the
action of $\Gamma$ on $X$, with the following property:\\ Let $ k,l,m\in \mathbb{N}$. If $k,l$ and $m$   satisfy $|k-l| \leq m \leq k+l$ with $f,g$ are in the group algebra $\mathbb{C}\Gamma$ are supported in $C_k$ and $C_l$ respectively, then
$$\|(f*g) \chi_{m}\|_{2}\leq c \|f \|_{2} \|g \|_{2}.$$
 If $k,l$ and $m$   satisfy $|k-l| >m$ or $m> k+l$, then $$\|(f*g) \chi_{m}\|_{2}=0.$$

Following the techniques in \cite[Lemma 1.3, Lemma 1.4]{H} and in \cite[Proposition 1.2.6]{Jo} we have:
 for $f$ supported in $C_{k}$ and for all $g$ supported in $C_{l}$ that
\begin{align*}
\|(f*g)\chi_{m}\|_{2}&\leq \sum_{l>0}\|(f*g\chi_{l})\chi_{m}\|_{2}\\
&\leq  C \|f \|_{2}\sum_{l=|k-m|}^{k+m} \|g\chi_{l} \|_{2}\\
&\leq  C \|f \|_{2}\sum_{l>0}^{2\min{(k,m)}} \|g\chi_{m+k-l} \|_{2}\\
&\leq  C \|f \|_{2}k^{\frac{1}{2}}\big(\sum_{l>0}^{2\min{(k,m)}} \|g\chi_{m+k-l} \|^{2}_{2}\big)^{\frac{1}{2}}.\\
\end{align*}
Thus
\begin{align*}
\|f*g\|_{2}^{2}&=\sum_{m>0}\|(f*g)\chi_{m}\|^{2}_{2}\\
&\leq C^{2}k\|f\|_{2} \sum_{m}\big(\sum_{l>0}^{2\min{(k,m)}} \|g\chi_{m+k-l} \|^{2}_{2}\big)\\
&\leq C^{2}k^{2} \|f \|_{2}^{2} \|g\|_{2}^{2}.
\end{align*}
Thus we obtain  for $f$ supported in $C_{k}$ and for all $g\in \ell^{2}(\Gamma)$ that
$$\sup _{\|g\|_{2}\leq 1}\|f *g \|\leq C k \|f\|_{2} ,$$ with some positive constant $C>0$.
\end{proof}
We use the equality on operators norms given by the amenability of the action on the boundary and we express the inequality of norm operators given in Proposition \ref{RD} in its dual form with the matrix coefficients associated with the boundary representation. 
We obtain:
\begin{prop}\label{RD}
Let $X$ be a complete simply connected Riemannian pinched negatively curved manifold and
let $\Gamma$ be a convex cocompact discrete group of isometries of $X$. Let $\nu$ be a $\Gamma$-quasi-invariant measure on $\partial X$ and consider $\pi_{\nu}$ its associated quasi-regular representation.
There exists $C>0$ such that for all unit verctors $\xi,\eta \in L^{2}(\partial X,\nu)$ we have
$$\sum_{\gamma \in C_{n}}|\langle \pi_{\nu}(\gamma)\xi,\eta \rangle |^{2}\leq Cn^{2}.$$
\end{prop}
\begin{proof}
Observe that it is sufficient to prove the above inequality only for positive vectors $\xi,\eta$ in $L^{2}(\partial X, \nu)$.\\
Using Proposition \ref{amenable} and Proposition \ref{RD} we have for a positive function $f$ supported in the annulus $C_{n}$  $$\|\pi_{\nu}(f)\|=\|\lambda_{\Gamma}(f)\|\leq C n\|f\|_{2}.$$
Consider $$f(\cdot)=\chi_{C_{n}}(\cdot) \langle \pi_{\nu}(\cdot) \xi,\eta \rangle,$$
with $\xi $ and $\eta$ two nonzero unit positive vectors in $L^{2}$, and notice that $f$ is a positive function on $\Gamma$ supported on $C_{n}$. 
We have 
\begin{align*}
0\leq\sum_{\gamma \in C_{n}}\langle \pi_{\nu}(\gamma)\xi,\eta \rangle^{2}&= \langle \pi_{\nu}(f)\xi,\eta \rangle\\
&\leq \|\pi_{\nu}(f)\|\\
&\leq Cn\|f\|_{2}\\
&\leq Cn \big( \sum_{\gamma \in C_{n}} \langle \pi_{\nu}(\gamma) \xi,\eta \rangle^{2} \big)^{1/2}.
\end{align*}
Divide each term of the above inequality by $\big( \sum_{\gamma \in C_{n}} \langle \pi_{\nu}(\gamma) \xi,\eta \rangle^{2} \big)^{1/2}$ and take the square to obtain:
$$ \sum_{\gamma \in C_{n}}\langle \pi_{\nu}(\gamma)\xi,\eta \rangle^{2} \leq C^{2} n^{2}.$$

\end{proof}

\subsection{Uniform boundedness}
We shall consider the operator : 
\begin{equation}\label{operator}
T^{n}_{x}:=c_{F,\Gamma}{\rm e}^{-\sigma n} \sum_{\gamma \in C_{n}(x) }{\rm e}^{d^{F}(x,\gamma x)}\frac{\pi_{x}(\gamma)}{\phi_{x}(\gamma)}.
\end{equation}
with $c_{\Gamma,F}$ given in (\ref{constante}) and recall that
 $C_{n}(x)=\lbrace n-1\leq  d(x,\gamma x)<n \rbrace$. Observe that $T^{n}_{x}$ is nothing else than $$T^{n}_{x}=\mathcal{M}_{x}^{n}(\textbf{1}_{\overline{X}}), $$
where $\mathbf{1}_{\overline{X}}$ denotes the unit function on the compact set $\overline{X}$.

The Harish-Chandra estimates are fundamental to prove the uniform boundedness of the sequence of operators defined above. The potential function $F$ is always assumed to be symmetric.
\begin{prop}\label{uniform}
We have  $\sup_{n}\| T^{n}_{x}\|<+\infty.$
\end{prop}
\begin{proof}
Pick $x\in X$, let $\nu:=\nu^{F}_{x}$ be a Patterson density of $(\Gamma,F)$ of dimension $\sigma$ and consider $\pi_{\nu}$ the quasi-regular representation associated. Then Proposition \ref{RD} implies for all unit vectors $\xi,\eta \in L^{2}(\partial X,\nu)$ we have $$ \sum_{\gamma \in C_{n}(x)} |\langle \pi_{\nu}(\gamma) \xi,\eta\rangle |^{2}\leq Cn^{2}.
 $$

Observe that Cauchy-Schwarz inequality implies that for all unit vectors $\xi,\eta,\xi^{'},\eta^{'}\in L^{2}(\partial X,\nu)$ we have
$$
 \sum_{\gamma \in C_{n}(x)} |\langle \pi_{\nu}(\gamma) \xi,\eta\rangle \langle \pi_{\nu}(\gamma) \xi',\eta'\rangle|\leq Cn^{2}.
 $$

Therefore for all unit vectors $ \xi,\eta, \xi',\eta'$ we have
\begin{align*}
Cn^{2}\geq &\sum_{\gamma \in C_{n}(x)} |\langle \pi_{\nu}(\gamma) \xi,\eta\rangle \langle \pi_{\nu}(\gamma) \xi',\eta'\rangle|\\
&=\sum_{\gamma \in C_{n}(x)} \frac{|\langle \pi_{\nu}(\gamma) \xi,\eta\rangle \langle \pi_{\nu}(\gamma) \xi',\eta'\rangle|}{\phi_{x}^{2}(\gamma)}\phi_{x}^{2}(\gamma)\\
\mbox{Proposition \ref{HCH}} &\geq C'\sum_{\gamma \in C_{n}(x)} d(x,\gamma x){\rm e}^{d^{F}(x,\gamma x)}{\rm e}^{-\sigma |\gamma|}\frac{|\langle \pi_{\nu}(\gamma) \xi,\eta\rangle \langle \pi_{\nu}(\gamma) \xi',\eta'\rangle|}{\phi_{x}^{2}(\gamma)}\\
&\geq n^{2}C' {\rm e}^{-\sigma n}\sum_{\gamma \in C_{n}(x)} {\rm e}^{d^{F}(x,\gamma x)}\frac{|\langle \pi_{\nu}(\gamma) \xi,\eta\rangle \langle \pi_{\nu}(\gamma) \xi',\eta'\rangle|}{\phi_{x}^{2}(\gamma)}\\
&\geq \frac{C'}{c_{\Gamma,F}}n^{2}\bigg( c_{\Gamma,F}{\rm e}^{-\sigma n}\sum_{\gamma \in C_{n}(x)} {\rm e}^{d^{F}(x,\gamma x)}\frac{|\langle \pi_{\nu}(\gamma) \xi,\eta\rangle \langle \pi_{\nu}(\gamma) \xi',\eta'\rangle|}{\phi_{x}^{2}(\gamma)}\bigg)
\end{align*}

Applying the above inequality for $\xi'=\eta'=1$ we obtain for all unit vectors $\xi$ and $\eta$ in $L^{2}(\partial X,\nu)$ :
$$Cn^{2}\geq  \frac{C'}{c_{\Gamma,F}}n^{2} |\langle T^{n}_{x}\xi ,\eta \rangle |.$$
Hence $$\sup_{n}\|T^{n}_{x}\|<\infty, $$
and the proof is done.

\end{proof}
\begin{remark}
Notice that Bader and Muchnik in \cite{BM} use a different method to prove uniform boundedness of the sequence of operators. Our method combining the property RD with the equality concerning the spectral radius gives another short proof of the uniform boundedness when the quasi-invariant measure is the Patterson-Sullivan measure class. \end{remark}

\begin{remark}
Notice also that this uniform bound for the Patterson-Sullivan measure class gives a sharp estimate of the spectral gap of $\lambda_{\Gamma}(\mu_{n})$ where $\mu_{n}$ is the probability measure on the groups supported over an annulus $C_{n}$ $$\mu_{n}=\frac{1}{|C_{n}|}\chi_{C_{n}} .$$ More specifically we obtain $$C^{-1}n {\rm e}^{{-\frac{1}{2}\sigma_{\Gamma}n}}\leq\| \lambda(\mu_{n})\|\leq C n {\rm e}^{{-\frac{1}{2}\sigma_{\Gamma}n}},$$ for some positive constant $C>0$ and where $\sigma_{\Gamma}$ is the usual critical exponent in the Patterson-Sullivan theory, with a potential $F=0$.

\end{remark}

\section{Analysis of matrix coefficients}\label{section5}
\subsection{Notation} Let $\Gamma$ be a discrete group of isometries of $X$ and let $\nu^{F}$ be a Patterson density of dimension $\sigma.$ Let $(d_{x})_{x\in X}$ be a family of visual metrics.\\
Let $U$ be a subset of $\partial X$ and $a>0$ be a positive real number and define $U_{x}(a)$ the subset of $\partial X$ as  
\begin{equation}\label{thick}
U_{x}(a)=\lbrace  v | \inf_{w\in A}d_{x}(v,w) <{\rm e}^{-a}\rbrace .
\end{equation}
We will write $U(a)$ instead of $U_{x}(a)$ once $x$ has been fixed. Recall that $\cap_{a>0}U(a)=\overline{U}$.
\\

In order to have Harish-Chandra's estimates available we pick $x\in CH(\Lambda_{\Gamma})$ for the rest of this section.

\begin{lemma}\label{fonda}
 Let $a>0$ be a positive real number, let $\gamma$ be in $\Gamma$ and let $w_{x}^{\gamma x} \in \partial X$. Consider the ball $\partial X$ defined as  $B_{a}=B(w_{x}^{\gamma x}, {\rm e}^{-a})$ and let $U$ be a Borel subset of $\partial X \backslash B_{a}$. There exists $C_{a}$ such that we have $$\frac{\langle \pi_{x}(\gamma)\mathbf{1}_{\partial X},\chi_{U}\rangle }{\phi_{x}(\gamma)} \leq \frac{C_{a}}{d(\gamma x,x)}\cdot$$
\end{lemma}

\begin{proof}

Define the following sets of indices 
 $$I=\{i \mbox{ such that } A_{i,R}(\gamma)\cap \partial X \backslash B_{a} \neq \varnothing \},$$
 and $$J=  \{i \mbox{ such that } {\rm e}^{-d(x,z_{i-1})}\geq {\rm e}^{-a}\}.$$ 
If $i$ is in $I$, then $A_{i,R}(\gamma)$ is not included in $B_{a}$. Since $A_{i,R}(\gamma) \subset B(w^{\gamma x}_{x},C{\rm e}^{-d(x,z_{i-1})})$ then $B(w^{\gamma x}_{x},C{\rm e}^{-d(x,z_{i-1})})$ cannot be included in $B_{a}$ where $C={\rm e}^{2\delta+R}$ (see Lemma \ref{lemme22}). This means that $i$ satisfies $C{\rm e}^{-d(x,z_{i-1})}\geq {\rm e}^{-a}$. 

There is only a finite number of $i$ such that $d(x,z_{i-1})\leq a+\log (C)=a+2\delta+R$. Hence by denoting $N_{a}:=|J|$ the cardinal of $J$, we obtain $|I|= |\{i \mbox{ such that } A_{i,R}(\gamma)\cap \partial X \backslash B_{\epsilon} \neq \varnothing \}|\leq |J|= N_{a}$.

Since $U$ is in $\partial X \backslash B_{a}$ we have:
\begin{align*}
\langle \pi(\gamma)\mathbf{1}_{\partial X}, \chi_{U}\rangle&\leq \sum_{i=1}^{N}\int_{A_{i,R}(\gamma)\cap \partial X \backslash B_{\epsilon}}{\rm e}^{{C_{F-\sigma,v}(x,\gamma x)}} d\nu_{x}^{F}(v)\\
&\leq \sum_{i\in I}\int_{A_{i,R}(\gamma)\cap \partial X \backslash B_{\epsilon}}{\rm e}^{{C_{F-\sigma,v}(x,\gamma x)}} d\nu_{x}^{F}(v)\\
 \mbox{Remark \ref{Aiestim}}&\leq C N_{a} {\rm e}^{{\frac{\sigma}{2}(d^{F}(x,\gamma x)-d(x,\gamma x))}}\\
\mbox{Left hand side inequality of Proposition \ref{HCH}}&\leq \frac{C_{a}}{d(x,\gamma x)}\phi_{x}(\gamma).
\end{align*}

\end{proof}
It turns out that the following results are very close to the results of \cite[Section 5]{BM}. We shall indicate all the minor modifications that we need to do to achieve Step 2. 

Recall the notation of a cone of basis $A\subset \partial X$ of vertex $x$ in $X$: $$\mathcal{C}_{x}(A).$$

\begin{prop}\label{deform}
Pick $x\in CH(\Lambda_{\Gamma})$ and let $\mu_{n}\in \ell^{1}(\Gamma)$ such that $$\sup_{n}\|\mu_{n}\|_{\ell^{1}}<+\infty,$$ and which satisfies $$\lim_{n\rightarrow +\infty}\mu_{n}(\gamma)=0,$$ for all $\gamma \in \Gamma$. Then for every Borel subset $U\subset  \partial X$ we have for all $a>0$ $$\limsup_{n\rightarrow+\infty} \sum_{\gamma \in \Gamma}\mu_{n}(\gamma)\frac{\langle \pi_{x}(\gamma)\mathbf{1},\chi_{U}\rangle}{\phi_{x}(\gamma)} \leq \limsup_{n \rightarrow+\infty} \sum_{\gamma \in \Gamma}\mu_{n}(\gamma)D_{\gamma  x}  ( \chi_{\mathcal{C}_{x}(U(a))}).$$
\end{prop}
\begin{proof}
Let $U$ be a Borel subset of $\partial X$ and let $a$ be a positive number and consider $U(a)$ (see Definition (\ref{thick})). Let $N_{0}$ be nonnegative integer. Consider the following partition of $\Gamma$: $$\Gamma=\Gamma_{1}\sqcup \Gamma_{2}\sqcup  \Gamma_{2}$$ with $$\Gamma_{1}=\lbrace\gamma \in \Gamma | d(x,\gamma  x)\leq N_{0} \rbrace$$ and  $$\Gamma_{2}=\lbrace \gamma \in \Gamma |\gamma x \in \mathcal{C}_{x}( U(a)) \rbrace \cap \Gamma_{1}^{c}$$ and  $$\Gamma_{3}=\lbrace \gamma \in \Gamma | \gamma x\notin \mathcal{C}_{x} (U(a)) \rbrace \cap \Gamma_{1}^{c}.$$   

Note that $ \gamma x\notin \mathcal{C}_{x} (U(a))$ is equivalent to $w^{\gamma x}_{x} \notin U(a)$. Therefore $U\cap B(w^{\gamma x}_{x},{\rm e}^{-a})=\varnothing$ so that Lemma \ref{fonda} is available. 
The proof follows now exactly the proof of \cite[Proposition 5.1]{BM} and \cite[Proposition 5.1]{Boy}.
\end{proof}

\subsection{Application of Paulin-Pollicott-Schapira's equidistribution Theorem}
The purpose of this section is to use Corollary \ref{coroPaulin} for computing the limit of the sequence of operator-valued measures $(\mathcal{M}_{x}^{n})_{n \in \mathbb{N}^{*}}$.\\
 We assume here that $\nu^{F}$ is a Patterson density of dimension $\sigma_{\Gamma,F}$ and that the Gibbs measure is mixing with respect to the geodesic flow. We shall prove:

\begin{prop}\label{applicationroblin}
Let $U,V,W \subset \partial {X}$ be Borel subsets such that $\nu^{F}_{x}(\partial U)=\nu^{F}_{x}(\partial V)=\nu^{F}_{x}(\partial W)=0$. Then we have:
$$\lim_{n\rightarrow +\infty} \langle \mathcal{M}^{n}_{x}(\chi_{\mathcal{C}_{x}(U)})\chi_{V},\chi_{W} \rangle=\nu^{F}_{x}(U \cap W)\nu^{F}_{x}(V).$$

\end{prop}

We need some lemmas to prepare the proof of this proposition.

\begin{lemma}\label{lemprepar1}
Let $U$ be a Borel subset of $\partial X$ with $\nu^{F}_{x}(\partial U)=0$ and let $W$ be a Borel subset of $\partial X$ such that $\nu^{F}_{x}(\partial W)=0$, satisfying $U\cap W=\varnothing$. Then we have $$\limsup_{n \rightarrow +\infty} \langle \mathcal{M} ^{n}_{x,\rho} (\chi_{\mathcal{C}_{x}(U)} ) \mathbf{1}_{\partial X} , \chi_{W} \rangle =0.$$
\end{lemma}

\begin{proof}
For all integers $n$ we have:
\begin{eqnarray*}
 \langle \mathcal{M} ^{n}_{x} (\chi_{\mathcal{C}_{x}(U)} ) \mathbf{1}_{\partial X} , \chi_{W} \rangle &=&c_{\Gamma,F}{\rm e}^{-\sigma_{\Gamma,F} n}  \sum_{\gamma \in C_{n}(x)} {\rm{e}}^{d^{F}(x,\gamma x)}D_{\gamma  x} (\chi_{\mathcal{C}_{x}(U)}) \frac{\langle \pi_{x}(\gamma)\mathbf{1}_{\partial X},\chi_{W} \rangle}{\phi_{x}(\gamma)} \\
  & = &\sum_{\gamma \in \Gamma} \mu_{n}(\gamma)\frac{ \langle \pi_{x}(\gamma)\mathbf{1}_{\partial X},\chi_{W} \rangle}{\phi_{x}(\gamma)},
\end{eqnarray*}

where the inequality follows from the fact that $\pi_{x}$ preserves the cone of positive functions, and where $$\mu_{n}(\gamma):=c_{\Gamma,F}{\rm e}^{-\sigma_{\Gamma,F} n}  {\rm{e}}^{d^{F}(x,\gamma x)}\chi_{C_{n}(x)}(\gamma) D_{\gamma  x} (\chi_{\mathcal{C}_{x}(U)}).$$
 Observe that Corollary \ref{coroPaulin} implies that 
 $$\sup_{n} \|\mu_{n}\|_{\ell^1}<+\infty.$$
Proposition \ref{deform} implies for $b>0$:
\begin{eqnarray*}
\limsup_{n \rightarrow + \infty} \langle \mathcal{M} ^{n}_{x} (\chi_{\mathcal{C}_{x}(U)} ) \mathbf{1}_{\partial X} , \chi_{W} \rangle &\leq& \limsup_{n \rightarrow +\infty} \sum_{\gamma \in \Gamma} \mu_{n}(\gamma)D_{\gamma  x}(\chi_{\mathcal{C}_{x}(W(b))}) \\
&=&\limsup_{n \rightarrow +\infty}  c_{\Gamma,F} {\rm e}^{-\sigma_{\Gamma,F} n}  \sum_{\gamma \in C_{n}(x)} {\rm{e}}^{d^{F}(x,\gamma x)}D_{\gamma  x} (\chi_{\mathcal{C}_{x}(U)\cap \mathcal{C}_{x}(W(b))})\\
&= &\limsup_{n \rightarrow +\infty}  c_{\Gamma,F}{\rm e}^{-\sigma_{\Gamma,F} n}  \sum_{\gamma \in C_{n}(x)} {\rm{e}}^{d^{F}(x,\gamma x)}D_{\gamma  x} (\chi_{ \mathcal{C}_{x}(U\cap W(b))}).  
 \end{eqnarray*}
 Note the general fact $\partial (A\cap B)\subset \partial A \cup \partial B$. Since all, but at most countably many of the sets $W(b)$ have zero measure boundary
 Corollary \ref{coroPaulin} implies that
 $$\limsup_{n \rightarrow +\infty} \langle \mathcal{M} ^{n}_{x} (\chi_{\mathcal{C}_{x}(U)} ) \mathbf{1}_{\partial X} , \chi_{W} \rangle \leq \nu^{F}_{x}\big(U\cap W(b)\big) \cdot$$
  With the hypothesis  $U\cap W=\varnothing$, we have by letting $b\rightarrow +\infty$ $$ \limsup_{n \rightarrow +\infty} \langle \mathcal{M} ^{n}_{x} (\chi_{\mathcal{C}_{x}(U)} ) \mathbf{1}_{\partial X} , \chi_{W} \rangle =0.$$
 \end{proof}

\begin{lemma}\label{lemprepar2}
Let $U$ be a Borel subset of $\partial X$ and let $V$ be a Borel subset of $\partial X$. For $a>0$ we have $$\limsup_{n\rightarrow +\infty} \langle \mathcal{M} ^{n}_{x} (\chi_{\mathcal{C}_{x}(U) } )  \chi_{V} ,\mathbf{1}_{\partial X} \rangle \leq \limsup_{n \rightarrow +\infty} c_{\Gamma,F} {\rm e}^{-\sigma_{\Gamma,F} n} \sum_{\gamma \in C_{n}(x)}  {\rm{e}}^{d^{F}(x,\gamma x)} D_{\gamma^{-1} x } (\chi_{\mathcal{C}_{x}(U)})  D_{\gamma  x} (\chi_{\mathcal{C}_{x}(V(a))}).$$ 
 \end{lemma}

\begin{proof}
We have for all integer $n$:
\begin{eqnarray*}
\langle \mathcal{M} ^{n}_{x} (\chi_{\mathcal{C}_{x}(U)} )  \chi_{V} ,\mathbf{1}_{\partial X} \rangle &=& \langle  \chi_{V} ,\mathcal{M} ^{n}_{x} (\chi_{\mathcal{C}_{x}(U)} )^{*} \mathbf{1}_{\partial X} \rangle\\
&= & c_{\Gamma,F} {\rm e}^{-\sigma_{\Gamma,F} n} \sum_{\gamma \in C_{n}(x)} {\rm{e}}^{d^{F}(x,\gamma x)}D_{\gamma ^{-1} x}(\chi_{\mathcal{C}_{x}(U)})\frac{\langle \pi_{x}(\gamma) \mathbf{1}_{\partial X} ,\chi_{V} \rangle}{\phi_{x}(\gamma)} \\
&\leq& \sum_{\gamma \in \Gamma} \mu_{n}(\gamma)\frac{\langle \pi_{x}(\gamma) \mathbf{1}_{\partial X} ,\chi_{V} \rangle}{\phi_{x}(\gamma)},
\end{eqnarray*}
with $$\mu_{n}(\gamma)= c_{\Gamma,F} {\rm e}^{-\sigma_{\Gamma,F} n}{\rm{e}}^{d^{F}(x,\gamma x)}\chi_{C_{n}(x)}(\gamma) D_{\gamma ^{-1} x}(\chi_{\mathcal{C}_{x}(U)}).$$

Applying Proposition \ref{deform} to $\mu_{n}$ defined above we obtain for all $a>0$: 
$$\limsup_{n\rightarrow +\infty} \langle \mathcal{M} ^{n}_{x} (\chi_{\mathcal{C}_{x}(U) } )  \chi_{V} ,\mathbf{1}_{\partial X} \rangle \leq \limsup_{n \rightarrow +\infty}c_{\Gamma,F} {\rm e}^{-\sigma_{\Gamma,F} n}\sum_{\gamma \in C_{n}(x)} D_{\gamma^{-1} x } (\chi_{\mathcal{C}_{x}(U)})  D_{\gamma  x} (\chi_{\mathcal{C}_{x}(V(a))}).$$
\end{proof}

\begin{lemma}\label{lemineg}
Let  $U,V,W \subset \partial {X}$ be Borel subsets such that $\nu^{F}_{x}(\partial U)=\nu^{F}_{x}(\partial V)=\nu^{F}_{x}(\partial W)=0$. Then 

$$\limsup_{n\rightarrow +\infty} \langle \mathcal{M}^{n}_{x}(\chi_{\mathcal{C}_{x}(U)})\chi_{V},\chi_{W} \rangle\leq \nu^{F}_{x}(U \cap W)\nu^{F}_{x}(V)\cdot $$
\end{lemma}
\begin{proof}
Let $a>0$ and $b>0$, and consider $V(a)$ and $W(b)$ such that $\nu^{F}_{x}(\partial W(b))=0=\nu^{F}_{x}(\partial V(a))$. Let $W(b)^{c}=\partial X \backslash W(b)$. Set $U_{1}=U \cap W(b)$ and  $U_{2}=U \cap W(b)^c$. Observe that $U_{1}\cap W(b)^c=\varnothing=U_{2}\cap W(b)$. It is easy to see that we can extend $U_{1}$ and $U_{2}$ to $\overline{X}$ by $\mathcal{C}_{x}(U_{1})$ and $\mathcal{C}_{x}(U_{2})$ such that $\mathcal{C}_{x}(U)=\mathcal{C}_{x}(U_{1})\sqcup \mathcal{C}_{x}(U_{2})$ since $U=U_{1}\sqcup U_{2}$.
We have:
\begin{eqnarray*}
\langle \mathcal{M}_{x}^{n}(\chi_{\mathcal{C}_{x}(U) })\chi_{V},\chi_{W}\rangle&=&
 \langle \mathcal{M}_{x}^{n}(\chi_{\mathcal{C}_{x}(U_{1})})\chi_{V},\chi_{W} \rangle+\langle \mathcal{M}_{x}^{n}(\chi_{\mathcal{C}_{x}(U_{2})})\chi_{V},\chi_{W} \rangle\\
 &\leq & \langle \mathcal{M}_{x}^{n}(\chi_{\mathcal{C}_{x}(U_{1})})\chi_{V},\mathbf{1}_{\partial X} \rangle+\langle \mathcal{M}_{x}^{n}(\chi_{\mathcal{C}_{x}(U_{2})})\mathbf{1}_{\partial X},\chi_{W} \rangle.
\end{eqnarray*}
 Applying Lemma \ref{lemprepar1} to the second term and Lemma \ref{lemprepar2} to the first term of the right hand side  inequality above, we obtain: $$\limsup_{n \rightarrow +\infty} \langle \mathcal{M}_{x}^{n}(\chi_{\mathcal{C}_{x}(U) })\chi_{V},\chi_{W}\rangle \leq \limsup_{n \rightarrow +\infty}c_{\Gamma,F} {\rm e}^{-\sigma_{\Gamma,F} n} \sum_{\gamma \in C_{n}(x)}D_{\gamma ^{-1} x} (\chi_{\mathcal{C}_{x}(U_{1})})  D_{\gamma  x} (\chi_{\mathcal{C}_{x}(V(a))} ).$$ 
  Then, since $\nu^{F}_{x}(\partial U_{1})=0=\nu^{F}_{x}(\partial V(a))$, Corollary \ref{coroPaulin} leads to 
  $$\limsup_{n \rightarrow +\infty} \langle \mathcal{M}_{x}^{n}(\chi_{\mathcal{C}_{x}(U) })\chi_{V},\chi_{W}\rangle \leq \nu^{F}_{x}(U\cap W(b))\nu^{F}_{x}(V(a)).$$
 Because the above inequality holds for all but at most countably many values of $a$ and $b$, by letting them go to $+\infty$ we obtain the required inequality.
 
\end{proof}

\begin{proof}[Proof of Proposition \ref{applicationroblin}]
By Lemma \ref{lemineg} it is sufficient to prove that $$\liminf_{n\rightarrow +\infty} \langle \mathcal{M}^{n}_{x}(\chi_{\mathcal{C}_{x}(U)})\chi_{V},\chi_{W} \rangle=\nu^{F}_{x}(U \cap W)\nu^{F}_{x}(V)\cdot $$
If $B$ is a Borel subset of $\partial X$, we set $B^{0}=B$ and $B^{1}=\partial X \backslash B$. We have 
\begin{eqnarray*}
\langle \mathcal{M}_{x}^{n}(\mathbf{1}_{\overline{ X}})\mathbf{1}_{\partial X},\mathbf{1}_{\partial X}\rangle 
&=&\langle \mathcal{M}_{x}^{n}(\chi_{\mathcal{C}_{x}(U^0)}+\chi_{\mathcal{C}_{x}(U^1)})\chi_{V^0}+\chi_{V^1},\chi_{W^0}+\chi_{W^1} \rangle \\
 &=& \sum_{i,j,k} \langle\mathcal{M}_{x}^{n}(\chi_{\mathcal{C}_{x}(U^i)})\chi_{V^j},\chi_{W^k} \rangle \\
 &=&  \langle\mathcal{M}_{x}^{n}(\chi_{\mathcal{C}_{x}(U)})\chi_{V},\chi_{W} \rangle+\sum_{i,j,k \neq (0,0,0)} \langle\mathcal{M}_{x}^{n}(\chi_{\mathcal{C}_{x}(U^i)})\chi_{V^j},\chi_{W^k} \rangle. 
\end{eqnarray*}
Since $\liminf _{n \rightarrow +\infty}\langle \mathcal{M}_{x}^{n}(\mathbf{1}_{\overline{ X}})\mathbf{1}_{\partial X},\mathbf{1}_{\partial X}\rangle=\lim _{n \rightarrow +\infty}\langle \mathcal{M}_{x}^{n}(\mathbf{1}_{\overline{ X}})\mathbf{1}_{\partial X},\mathbf{1}_{\partial X}\rangle = \|\nu^{F}_{x}\|^{2}$ we have:
\begin{eqnarray*}
\|\nu^{F}_{x}\|^{2} &\leq& \liminf_{n \rightarrow +\infty}\langle \mathcal{M}_{x}^{n}(\chi_{\mathcal{C}_{x}(U)})\chi_{V},\chi_{W} \rangle+\sum_{i,j,k \neq (0,0,0)}  \limsup_{n \rightarrow + \infty} \langle \mathcal{M}_{x}^{n}(\chi_{\mathcal{C}_{x}(U^i)})\chi_{V^j},\chi_{W^k} \rangle \\
&\leq &  \limsup_{n \rightarrow +\infty}\langle \mathcal{M}_{x}^{n}(\chi_{\mathcal{C}_{x}(U)})\chi_{V},\chi_{W} \rangle+\sum_{i,j,k \neq (0,0,0)}  \limsup_{n \rightarrow + \infty} \langle \mathcal{M}_{x}^{n}(\chi_{\mathcal{C}_{x}(U^i)})\chi_{V^j},\chi_{W^k} \rangle \\
& \leq &\sum_{i,j,k } \nu^{F}_{x}(U^{i} \cap W^k) \nu^{F}_{x}(V^j) \\
&=& \|\nu^{F}_{x}\|^{2},
\end{eqnarray*}
where the last inequality comes from Lemma \ref{lemineg}. Hence the inequalities of the above computation are equalities, so $$\liminf_{n\rightarrow +\infty} \langle \mathcal{M}^{n}_{x}(\chi_{\mathcal{C}_{x}(U)})\chi_{V},\chi_{W} \rangle =\nu^{F}_{x}(U \cap W)\nu^{F}_{x}(V)=\limsup_{n\rightarrow +\infty} \langle \mathcal{M}^{n}_{x}(\chi_{\mathcal{C}_{x}(U)})\chi_{V},\chi_{W} \rangle $$
and the proof is done.

\end{proof}
\section{Conclusion}\label{conclusion}

\subsection{Standard facts about Borel subsets of measure zero frontier}
Recall two standard facts about measure theory:
\begin{lemma}\label{mesurenullebord}
Assume that $(Z,d,\mu)$ is a metric measure space. Then the $\sigma$-algebra generated by the Borel subsets with measure zero frontier generates the Borel $\sigma$-algebra.
\end{lemma}

Let $\chi_{A}$ be the characteristic function of a Borel subset $A$ of $\partial X$.

\begin{lemma}\label{denseL2}
Assume that $(Z,d,\mu)$ is a metric measure space such that $\mu$ is regular. Then the closure of the subspace spanned by the characteristic functions of Borel subsets having zero measure frontier is $$\overline{Span\lbrace \chi_{A} | \mu(\partial A)=0 \rbrace}^{L^2}=L^{2}(Z,\mu).$$
\end{lemma}
The proof is a direct application of the Lebesgue differentiation theorem.

\subsection{Proofs}

\begin{proof}[Proof of Theorem A] 
Let $\nu^{F}$ be a $\Gamma$-invariant Gibbs conformal density of dimension $\sigma_{\Gamma,F}$ with $\widetilde{F}$ a symmetric potential function and $\Gamma$ convex cocompact. Let $x$ be in the $CH(\Lambda_{\Gamma})$ and consider $\pi_{x}$ associated with $\nu^{F}_{x}$.
 There are two steps. \\

\textbf{Step 1: $(\mathcal{M}_{x}^{n})_{n\geq N}$ is uniformly bounded.} Note that the norm of operators of $\mathcal{M}_{x}^{n} $ is less or equal than the norm of $\mathcal{M}_{x}^{n}(\textbf{1}_{ \overline{X} })$. Recall that  $$\mathcal{M}_{x}^{n}(\textbf{1}_{ \overline{X} })=T^{n}_{x}$$ where $T^{n}_{x}$ is the sequence of operators defined in (\ref{operator}).
Proposition  \ref{uniform} completes the first step.

\textbf{Step 2: Computation of the limit of $(\mathcal{M}_{x}^{n})_{n \in\mathbb{ N}^{*}}$}. As in \cite{BM} and in \cite{Boy}, the sequence  $(\mathcal{M}_{x}^{n})_{n\in \mathbb{N}^{*}}$  has actually one accumulation point (with respect to the weak* topology of  $\mathcal{L}\big( C(\overline{X}),\mathcal{B}(L^{2}(\partial X,\nu^{F}_{x}))\big)$ that we denote by $\mathcal{M}^{\infty}_{x}$. We shall compute it:

  Since we assume that the Gibbs measure is mixing it follows from Proposition \ref{applicationroblin} and from the definition (\ref{mesurelimite}) of $\mathcal{M}_{x}$ that for all Borel subsets $U,V,W\subset \partial X$ satisfying $\nu^{F}_{x}(\partial U)=\nu^{F}_{x}(\partial V)=\nu^{F}_{x}(\partial W)=0$ we have $$\langle \mathcal{M}^{\infty}_{x}(\chi_{\mathcal{C}_{x}(U)})\chi_{V},\chi_{W}\rangle=\nu^{F}_{x}(U \cap W)\nu^{F}_{x}(V)=\langle \mathcal{M}_{x}(\chi_{\mathcal{C}_{x}(U)})\chi_{V},\chi_{W}\rangle.$$
  Observe also that the above equality holds for all balls of the space $X$ instead of $C_{x}(U)$ and everything is null in this case. Since $\lbrace \mathcal{C}_{x}(U)| U \subset \partial X \mbox{ such that } \nu^{F}_{x}(\partial U)=0 \rbrace$ together with the balls of $X$ generate the Borel $\sigma$-algebra of $\overline{X}$ and since the equality holds for all Borel subsets having zero measure boundary Lemma \ref{denseL2}  completes the proof. 
\end{proof}

\begin{proof}[Proof of Corollary B]
Observe that $\mathcal{M}_{x}(\textbf{1}_{\overline{X}})$ is the orthogonal projection onto the space of constant functions and apply the definition of weak$^{*}$ convergence to the triple $(\textbf{1}_{\overline{X}},  \xi , \eta)$ for $\xi,\eta \in L^{2}(\partial X, \nu^{F}_{x})$.
\end{proof}

\begin{proof}[Proof of Corollary C]
Since $(\pi_{\nu_{x}^F})_{x\in X}$ are unitarily equivalent, it suffices to prove irreducibility for some $\pi_{\nu_{x}^{F}}$ with $x$ in $ X$. We pick $x$ in $CH(\Lambda_{\Gamma})$. Since $\widetilde{F}$ is cohomologuous to a symmetric potential by Lemma \ref{cohomo} we can assume that $\widetilde{F}$ itself is symmetric. Therefore Theorem A shows that the vector $\mathbf{1}_{\partial X}$ is cyclic for the representation $\pi_{\nu^{F}_{x}}$ by applying the weak$^{*}$ convergence to the triple $(f,  \mathbf{1}_{\partial X} , \eta)$. 
Moreover, Corollary B shows that the orthogonal projection onto the space of constant functions is in the von Neumann algebra associated with $\pi_{\nu^{F}_{x}}$. Thus, the argument of \cite[Lemma 6.1]{LG} completes the proof. 

\end{proof}

Before giving the proof of Theorem D we recall that an operator $T\in \mathcal{B}(\mathcal{H})$ -where $\mathcal{H}=L^{2}(X,m)$ is a Hilbert space  for some measure space $(X,m)$- is a \emph{positive operator} if it preserves $\mathcal{H}^{+}$ the cone of positive functions. For example, any quasi-regular representation is a positive operator as well as the operators we consider in (\ref{operator}). 

\begin{proof}[Proof of Theorem D]
The implications: $(2) \Rightarrow  (3) \Rightarrow  (4)  \Rightarrow  (1)$ follow from Proposition \ref{equivcohomo}. \\We only have to prove $(1) \Rightarrow(2)$. We follow a standard method, see for example \cite[Lemma 7.3]{LG}:\\
Let $\pi_{F}:=\pi_{\nu_{x}^{F}}$ and $\pi_{G}:=\pi_{\nu_{x}^{G}}$ be equivalent unitary representations associated with $\nu_{x}^{F}$ and $\nu_{x}^{G}$, with $\widetilde{F}$ and $\widetilde{G}$ two symmetric potentials.  There exists $U$ a unitary operator from $L^{2}(\partial X, \nu^{F}_{x})$ to $L^{2}(\partial X, \nu^{G}_{x}) $ satisfying

$$U\pi_{F}=\pi_{G}U.$$ The map $$\Phi: T\in W^{*}_{\pi_{G}}(\Gamma) \mapsto U^{*} T U \in W^{*}_{\pi_{F}}(\Gamma) $$ is a spatial isomorphism of von Neumann algebras. It follows from the irreducibility of these representations (Corollary C) that the von Neumann algebras $W^{*}_{\pi_{F}}(\Gamma)= \mathcal{B}(L^{2}(\partial X, \nu^{F}_{x}))$ and $W^{*}_{\pi_{G}}(\Gamma)= \mathcal{B}(L^{2}(\partial X, \nu^{G}_{x}))$. 
Consider now the maximal abelian von Neumann algebras\\ $L^{\infty}(\partial X, \nu^{F})\subset \mathcal{B}(L^{2}(\partial X, \nu^{F}_{x}))$ and $L^{\infty}(\partial X, \nu^{G})\subset \mathcal{B}(L^{2}(\partial X, \nu^{G}_{x}))$ acting on $L^{2}$ by multiplication. Now observe that
the set of projections $$\lbrace p\in  \mathcal{B}(L^{2}(\partial X, \nu^{F}_{x})) \mbox{ such that $p$ and $1-p$ are orthogonal positive projections} \rbrace $$ is equal to the set 
$$ \lbrace \chi_{B}\mbox{ where $B$ is a Borel subset of $\partial X$}\rbrace .$$
 
Since the isomorphism $\Phi$ preserves the cone of positive operators 
 and since $L^{\infty}(\partial X,\nu^{F})$ is generated by its projections $\chi_{B}$ with $B$ Borel subsets, the automorphism $\Phi$ restricts to an algebra isomorphism from $\Phi:L^{\infty}(\partial X ,\nu^{G})\rightarrow L^{\infty}(\partial X, \nu^{F})$. It is well known that there exists $\varphi: (\partial X,\nu^{F})\rightarrow (\partial X,\nu^{G})$  a measure class preserving Borel isomorphism such that $\Phi(f)=f\circ \varphi $ for all $f\in L^{\infty}(\partial X ,\nu^{G})$. Therefore $\nu^{G}$ and $\nu^{F}$ are in the same class.
 \end{proof}

\end{document}